\documentclass[10pt]{amsart}
\usepackage{amsmath,amsfonts,amssymb,color,a4,epsfig,graphics}
\usepackage[english]{babel}
\usepackage{enumerate, mathrsfs}
\usepackage{xcolor}
\usepackage{verbatim}

\usepackage{tikz}

\usetikzlibrary{positioning}

\usepackage{url}
\usepackage[a4paper, margin=3 cm]{geometry}
\usepackage[
  colorlinks=true,
  linkcolor=blue,
  citecolor=red,
  urlcolor=blue
]{hyperref}
\usepackage{doi}
\usepackage[utf8]{inputenc}
\usepackage[T1]{fontenc}
\usepackage{enumitem}      
\usepackage{booktabs}
\usepackage{pgfplots}
\usepgfplotslibrary{groupplots}
\pgfplotsset{compat=1.18}
\usepackage{pgfplotstable}
\usepackage{graphicx}
\usepackage{hyperref}
\usepackage{amsmath,amssymb}
\usepackage{caption}      
\usepackage{subcaption}   
\usepackage{algorithm}
\usepackage{algorithmic}

\def\build#1_#2^#3{\mathrel{\mathop{\kern 0pt#1}\limits_{#2}^{#3}}}
\usetikzlibrary{arrows}

\newcommand{\R}{{\mathbb{R}}}
\newcommand{\C}{{\mathbb{C}}}

\newcommand{\Ran}{\mathrm{Ran}}

\newcommand{\asc}{\mathrm{asc}}
\newcommand{\dsc}{\mathrm{dsc}}

\newcommand{\codim}{\mathrm{Codim}}
\newcommand{\dist}{\mathrm{dist}}

\numberwithin{equation}{section}

\makeatletter
\def\@subjclass{{\bfseries 2020 Mathematics Subject Classification. }}
\makeatother

\newtheorem{theorem}{Theorem}[section]
\newtheorem{proposition}[theorem]{Proposition}
\newtheorem{lemma}[theorem]{Lemma}
\newtheorem{remark}[theorem]{Remark}
\newtheorem{example}[theorem]{Example}

\newtheorem{corollary}[theorem]{Corollary}
\newtheorem{conjecture}[theorem]{Conjecture}

\theoremstyle{plain}

\begin{document}

\title[Ascent--Descent Stability under Strong Resolvent]%
{Sharp Ascent--Descent Spectral Stability under Strong Resolvent Convergence}
\author{Marwa Ennaceur}
 \address{ Department of Mathematics, College of Science, University of Ha'il, Hail 81451, Saudi Arabia}
\email{\tt mar.ennaceur@uoh.edu.sa}


\makeatletter
\makeatother
\begin{abstract}
We establish sharp stability results for the ascent and descent spectra under strong resolvent convergence (SRS), a natural framework for finite element approximations of non-selfadjoint and singularly perturbed operators. The key quantitative hypothesis is the reduced minimum modulus $\gamma(T-\lambda)>0$, which guarantees closed range and enables the transfer of the Kaashoek--Taylor criteria via gap convergence of operator graphs. At the essential level, B--Fredholm theory extends stability to powers $(T-\lambda)^m$ provided $\gamma((T-\lambda)^j)>0$ for all $1\le j\le m$. We introduce a computable finite-element diagnostic $\gamma_h = \sigma_{\min}(M^{-1/2}(A_h-\lambda M)M^{-1/2})$, which serves as a practical surrogate for $\gamma(T-\lambda)$ and remains uniformly positive even in convection-dominated regimes when stabilized schemes (e.g., SUPG) are employed. Numerical experiments confirm that $\liminf_{h\to0}\gamma_h>0$ is both necessary and sufficient for spectral stability, while a Volterra-type counterexample demonstrates the indispensability of the closed-range condition for powers. The analysis clarifies why norm resolvent convergence fails for rough or singular limits, and how SRS—combined with quantitative control of $\gamma_h$—rescues ascent--descent stability in realistic computational settings.
\end{abstract}
\maketitle
\tableofcontents

\section{Introduction}\label{sec:intro}

Consider the one-dimensional transport operator  
\[
Lu = u', \quad u(0) = 0,\quad x \in (0,1).
\]  
Its continuous realization \(L\) on \(L^2(0,1)\) has  closed range  and  finite ascent: \(\operatorname{asc}(L) = 1\). Yet, if we discretize \(L\) with  second-order central differences  on a uniform mesh of size \(h\), the resulting matrix \(A_h\) exhibits a paradox: its eigenvalues cluster near zero—suggesting spectral convergence—yet its  discrete ascent diverges, \(\operatorname{asc}(A_h) = \infty\) for all \(h > 0\). This \emph{spectral illusion} hides a catastrophic breakdown in the underlying algebraic structure: the range of \(A_h\) fails to be closed in the limit.

This paradox illustrates a general principle: \emph{mesh refinement alone does not guarantee stability of fine spectral invariants}. To understand which discretizations succeed and which fail, we must examine the algebraic structure encoded by ascent and descent.

These invariants quantify the stabilization of kernel and range chains under iteration:
\[
\ker(T - \lambda) \subset \ker(T - \lambda)^2 \subset \cdots,
\qquad
\operatorname{Ran}(T - \lambda) \supset \operatorname{Ran}(T - \lambda)^2 \supset \cdots.
\]

In stark contrast, the  first-order upwind scheme  preserves the correct algebraic behavior: both \(\operatorname{asc}(A_h) = 1\) and the discrete range remain uniformly closed. What distinguishes these two discretizations? The answer lies in a single, computable quantity:
\[
\gamma_h := \sigma_{\min}\!\bigl(M^{-1/2}(A_h - \lambda M)M^{-1/2}\bigr),
\]
the discrete reduced minimum modulus. For the central scheme, \(\gamma_h \sim C h \to 0\); for upwind, \(\gamma_h \ge c > 0\) uniformly in \(h\). Thus, \(\gamma_h\) acts as a sharp, practical diagnostic: it predicts whether the fine spectral invariants-ascent and  descent-
survive mesh refinement.

These invariants quantify the stabilization of kernel and range chains under iteration:
\[
\ker(T-\lambda) \subset \ker(T-\lambda)^2 \subset \cdots, \qquad
\operatorname{Ran}(T-\lambda) \supset \operatorname{Ran}(T-\lambda)^2 \supset \cdots.
\]
For a closed operator \(S\), the \emph{ascent} and \emph{descent} are
\[
\operatorname{asc}(S) = \inf\{m \ge 0 : \ker S^m = \ker S^{m+1}\}, \quad
\operatorname{dsc}(S) = \inf\{m \ge 0 : \operatorname{Ran}(S^m) = \operatorname{Ran}(S^{m+1})\},
\]
with value \(\infty\) if the chain never stabilizes. Their divergence defines the ascent and descent spectra:
\[
\sigma_{\asc}(T) = \{\lambda : \operatorname{asc}(T-\lambda) = \infty\}, \quad
\sigma_{\dsc}(T) = \{\lambda : \operatorname{dsc}(T-\lambda) = \infty\}.
\]

In finite dimensions, these indices are fragile: a nilpotent matrix \(S\) can have \(\operatorname{asc}(S)=3\), yet an arbitrarily small perturbation \(S+\varepsilon I\) collapses it to 0. In infinite dimensions,  stability under approximation is therefore nontrivial. We show that, for finite element discretizations of non-selfadjoint or singularly perturbed operators (e.g., convection–diffusion with \(\varepsilon \ll 1\) or Schrödinger with \(L^\infty\) potentials),  ascent-descent stability  holds if and only if
\[
\liminf_{h \to 0} \gamma_h > 0.
\]

The continuous counterpart of \(\gamma_h\) is the reduced minimum modulus
\[
\gamma(S) := \inf\bigl\{ \|Sx\| : x \in \mathcal{D}(S),\ \operatorname{dist}(x,\ker S) = 1 \bigr\},
\]
which satisfies \(\gamma(S) > 0\) iff \(\operatorname{Ran}(S)\) is closed. This condition prevents the collapse observed in perturbations and enables the transfer of the Kaashoek-Taylor criteria \cite{Kaashoek1967,Taylor1966}:
\[
\operatorname{asc}(S) < \infty \iff \exists m:\ \operatorname{Ran}(S^m) \cap \ker S = \{0\}, \quad
\operatorname{dsc}(S) < \infty \iff \exists m:\ \operatorname{Ran}(S) + \ker(S^m) = H.
\]

Our convergence framework is  strong resolvent convergence (SRS)-natural for operators with rough coefficients or vanishing diffusion, where norm resolvent convergence fails. Under SRS and the quantitative hypothesis \(\gamma((T-\lambda)^j) > 0\) for all \(1 \le j \le m\), the graphs of powers converge in the gap topology, and the Kaashoek–Taylor conditions—and hence ascent/descent—are preserved. A Volterra-type counterexample confirms that the condition on all intermediate powers  is indispensable [Appendix].

At the essential level,  B-Fredholm theory  extends these results to compact perturbations, ensuring stability of the essential ascent and descent spectra \(\sigma_{\asc}^{\mathrm{e}}, \sigma_{\dsc}^{\mathrm{e}}\).

The paper is structured to move from  concrete failure  to  abstract mechanism and back to  computational validation:
\begin{itemize}
  \item Section~\ref{sec:Background} recalls functional-analytic foundations.
  \item Section~\ref{sec:Main} proves sharp stability results under SRS.
  \item Section~\ref{sec:Applications} validates the theory on Schrödinger and convection–diffusion operators, including SUPG stabilization and the central-difference failure.
\end{itemize}

Figure~\ref{fig:asc-desc-diagram} (at the end of this section) summarizes the logical chain: from numerical instability, through the diagnostic \(\gamma_h\), to the subspace criteria and convergence mechanism that guarantee robustness.

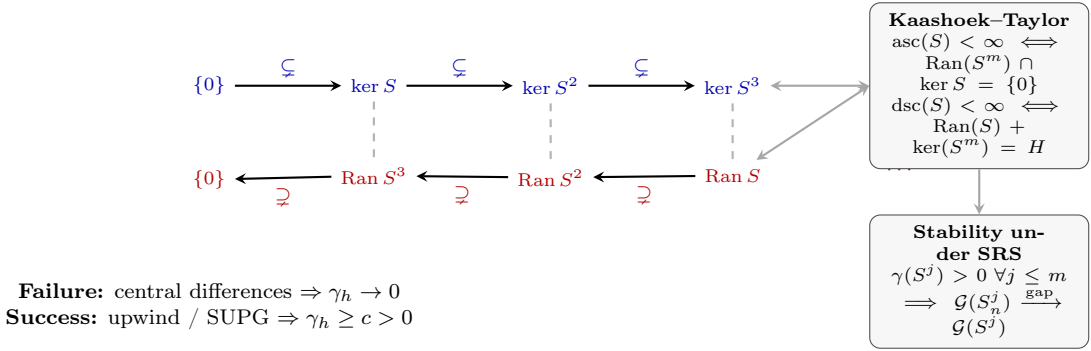
\begin{figure}[ht]
\centering
\begin{tikzpicture}[
    >=stealth,
    node distance=0.75cm and 1.4cm,
    every node/.style={font=\scriptsize},
    kerneltext/.style={text=blue!70!black},
    rangetext/.style={text=red!70!black},
    box/.style={rectangle, draw=black!70, fill=black!3, rounded corners, inner sep=4pt, align=center, text width=2.6cm}
]

\node[kerneltext] (K0) {$\{0\}$};
\node[kerneltext, right=of K0] (K1) {$\ker S$};
\node[kerneltext, right=of K1] (K2) {$\ker S^2$};
\node[kerneltext, right=of K2] (K3) {$\ker S^3$};
\node[right=of K3] (Kdots) {$\cdots$};

\node[rangetext, below=of K0] (R0) {$\{0\}$};
\node[rangetext, below=of K1] (R1) {$\operatorname{Ran} S^3$};
\node[rangetext, below=of K2] (R2) {$\operatorname{Ran} S^2$};
\node[rangetext, below=of K3] (R3) {$\operatorname{Ran} S$};
\node[rangetext, below=of Kdots] (Rdots) {$\cdots$};

\foreach \i/\j in {K0/K1, K1/K2, K2/K3}
  \draw[->, thick, kerneltext] (\i) -- node[above] {$\subsetneq$} (\j);
\foreach \i/\j in {R3/R2, R2/R1, R1/R0}
  \draw[->, thick, rangetext] (\i) -- node[below] {$\supsetneq$} (\j);

\draw[dashed, thick, gray!60] (K1) -- (R1);
\draw[dashed, thick, gray!60] (K2) -- (R2);
\draw[dashed, thick, gray!60] (K3) -- (R3);

\node[box, right=1.3cm of K3] (KT) {
    \textbf{Kaashoek–Taylor} \\
    $\operatorname{asc}(S)<\infty \iff$ \\
    $\operatorname{Ran}(S^m)\cap\ker S=\{0\}$ \\
    $\operatorname{dsc}(S)<\infty \iff$ \\
    $\operatorname{Ran}(S)+\ker(S^m)=H$
};

\node[box, below=0.6cm of KT] (SRS) {
    \textbf{Stability under SRS} \\
    $\gamma(S^j)>0\ \forall j\le m$ \\
    $\implies \mathcal{G}(S_n^j)\xrightarrow{\text{gap}} \mathcal{G}(S^j)$
};

\draw[<->, thick, gray!70] (K3) -- (KT.west);
\draw[<->, thick, gray!70] (R3) -- (KT.west);
\draw[->, thick, gray!70] (KT) -- (SRS);

\node[below=1.0cm of R0, align=center, font=\footnotesize] (failure) {
    \textbf{Failure:} central differences $\Rightarrow \gamma_h \to 0$ \\
    \textbf{Success:} upwind / SUPG $\Rightarrow \gamma_h \ge c > 0$
};

\end{tikzpicture}
\caption{From numerical failure to abstract stability. The kernel/range chains stabilize at $\operatorname{asc}(S)$ and $\operatorname{dsc}(S)$. The Kaashoek–Taylor criteria link these to subspace transversality. Under strong resolvent convergence (SRS), graph convergence—and hence stability-is guaranteed precisely when $\gamma((T-\lambda)^j) > 0$ for all intermediate powers. The computable quantity $\gamma_h$ predicts success or failure in practice.}
\label{fig:asc-desc-diagram}
\end{figure}

\section{Background and Notation}\label{sec:Background}

Throughout, $X$ denotes a complex Banach space and $H$ a complex Hilbert space.
We write $\mathcal{B}(X)$ for the bounded operators on $X$, $\mathcal{K}(X)$ for the ideal of compact operators,
and $q:\mathcal{B}(X)\to \mathcal{B}(X)/\mathcal{K}(X)$ for the Calkin quotient.
For a (possibly unbounded) closed operator $S$, we denote its domain, kernel, and range by
$\mathcal{D}(S)$, $N(S)$, and $\operatorname{Ran}(S)$, respectively.
The spectrum and resolvent set are denoted by $\sigma(S)$ and $\rho(S)$.
On Hilbert spaces, $E^\perp$ stands for the orthogonal complement of a closed subspace $E$.

\subsection{Ascent, descent, and the reduced minimum modulus}\label{subsec:asc-des-gamma}

For a closed operator $S$, the kernel and range chains under iteration are
\[
\ker S \subset \ker S^2 \subset \cdots, \qquad
\operatorname{Ran} S \supset \operatorname{Ran} S^2 \supset \cdots.
\]
The \emph{ascent} $\operatorname{asc}(S)$ is the smallest integer $m\ge0$ such that
$\ker S^m = \ker S^{m+1}$ (with $\operatorname{asc}(S)=\infty$ if no such $m$ exists);
the \emph{descent} $\operatorname{dsc}(S)$ is defined analogously for the ranges.
The corresponding spectra are
\[
\sigma_{\mathrm{asc}}(T) := \{\lambda : \operatorname{asc}(T-\lambda)=\infty\}, \qquad
\sigma_{\mathrm{dsc}}(T) := \{\lambda : \operatorname{dsc}(T-\lambda)=\infty\}.
\]

A key quantitative tool is the \emph{reduced minimum modulus}
\[
\gamma(S) := \inf\{ \|Sx\| : x \in \mathcal{D}(S),\ \operatorname{dist}(x, \ker S) = 1 \},
\]
which satisfies $\gamma(S) > 0$ if and only if $\operatorname{Ran}(S)$ is closed.
This condition underpins the stability of ascent and descent under perturbations:
it guarantees that the Kaashoek–Taylor characterizations—linking $\operatorname{asc}(S)$ and $\operatorname{dsc}(S)$ to
the triviality of intersections $\operatorname{Ran}(S^m)\cap\ker S$ and the closedness of sums $\operatorname{Ran}(S)+\ker(S^m)$—are preserved under approximation.
In the context of PDE discretizations, strong resolvent convergence (SRS) provides the natural convergence framework,
and the computable discrete surrogate $\gamma_h$ serves as a practical diagnostic for the continuous hypothesis $\gamma(T-\lambda)>0$.

\subsection{Fredholm and B–Fredholm theory}

A bounded operator $S$ is:\begin{itemize}
\item \emph{upper semi-Fredholm} if $\operatorname{Ran}S$ is closed and $\dim N(S)<\infty$,
\item \emph{lower semi-Fredholm} if $\operatorname{Ran}S$ is closed and $\operatorname{codim}\operatorname{Ran}S<\infty$,
\item \emph{Fredholm} if it is both upper and lower semi-Fredholm.\end{itemize}
By Atkinson's theorem \cite[Thm.~1.4.5]{Douglas1998}, $S$ is Fredholm iff $q(S)$ is invertible in the Calkin algebra.
An operator is \emph{Browder} if it is Fredholm and has finite ascent and descent
(equivalently, $q(S)$ is Drazin invertible in the Calkin algebra) \cite[Ch.~6]{Aiena2004}.

An operator $S$ is \emph{B--Fredholm} if there exists $m\ge1$ such that $\operatorname{Ran}(S^{m})$ is closed and
the restriction $S|_{\operatorname{Ran}(S^{m})}$ is Fredholm
(semi--B--Fredholm if the restriction is merely semi-Fredholm) \cite{Berkani1999,Berkani2002}.
This yields the characterizations
\[
\sigma_{\mathrm{aesc}}(T) = \{\lambda : T-\lambda \notin \Phi_{\ell}\}, \qquad
\sigma_{\mathrm{desc}}(T) = \{\lambda : T-\lambda \notin \Phi_{r}\},
\]
where $\Phi_{\ell}$ and $\Phi_{r}$ denote the classes of upper and lower semi-Fredholm operators.
These essential spectra are stable under compact perturbations and form the natural setting for ascent–descent stability in the non-selfadjoint case.

\subsection{Graph convergence and strong resolvent convergence}

The graph of a closed operator $S$ on $H$ is
\[
\mathcal{G}(S) := \{(x,Sx) : x\in\mathcal{D}(S)\} \subset H\times H.
\]
The \emph{gap distance} between closed subspaces $E,F\subset H\times H$ is
$\operatorname{gap}(E,F):=\|P_E-P_F\|$, where $P_E$ is the orthogonal projection onto $E$.
We write $\mathcal{G}(S_n)\to\mathcal{G}(S)$ in the gap topology when
$\|P_{\mathcal{G}(S_n)}-P_{\mathcal{G}(S)}\|\to0$.

For closed operators $T_n,T$ on $H$, \emph{strong resolvent convergence} (SRS) means that
for some (hence any) $\mu\in\rho(T)\cap\bigcap_n\rho(T_n)$,
\[
(T_n-\mu)^{-1}x \longrightarrow (T-\mu)^{-1}x \quad\text{for all }x\in H.
\]
In the selfadjoint or $m$-sectorial case, SRS follows from Mosco convergence of the associated
quadratic forms \cite{Kato1995,Mosco1969,Attouch1984}.
In general, SRS implies gap convergence of graphs \cite[Sec.~IV]{Kato1995}.
Crucially, if $\gamma(S^j)>0$ for $1\le j\le m$, then gap convergence propagates to powers:
$\mathcal{G}(S_n^j)\to\mathcal{G}(S^j)$ (see Lemma~\ref{lem:powers-graphs}).

\subsection{Finite element framework}

Let $V_h\subset H_0^1(\Omega)$ be a conforming finite element space with basis $\{\varphi_i\}$.
The \emph{mass matrix} $M$ and \emph{stiffness--convection matrix} $A$ are defined by
\[
M_{ij} = \int_\Omega \varphi_i\varphi_j\,dx, \qquad
A_{ij} = a(\varphi_j,\varphi_i),
\]
where $a(\cdot,\cdot)$ is the sesquilinear form associated with the underlying PDE.
The matrix $M$ is Hermitian positive definite and induces the discrete inner product
$\langle u,v\rangle_M := u^\ast M v$ and norm $\|u\|_M := (u^\ast M u)^{1/2}$.
The dual norm is $\|f\|_{M^{-1}} := (f^\ast M^{-1}f)^{1/2}$, satisfying
$\langle f,u\rangle \le \|f\|_{M^{-1}}\|u\|_M$.

For the generalized eigenvalue problem $A_h u = \mu M u$, the
\emph{discrete reduced minimum modulus} is
\[
\gamma_h := \inf_{u\ne0} \frac{\|(A_h - \lambda M)u\|_{M^{-1}}}{\|u\|_M}
          = \sigma_{\min}\!\bigl(M^{-1/2}(A_h - \lambda M)M^{-1/2}\bigr).
\]

In the $M$--selfadjoint case ($A_h^\ast M = M A_h$), one has
$\gamma_h = \operatorname{dist}\bigl(\lambda,\sigma(A_h,M)\bigr)$.
In the non-normal case, the $M$--numerical range
\[
W_M(A_h) := \bigl\{ u^\ast A_h u / u^\ast M u : u\ne0 \bigr\}
\]
satisfies the lower bound
\[
\gamma_h \ge \operatorname{dist}\bigl(\lambda, W_M(A_h)\bigr).
\]
Thus, uniform positivity of $\gamma_h$ across meshes provides a verifiable discrete surrogate
for the continuous hypothesis $\limsup_h \gamma(S_h) > 0$ used in the stability theory.
\begin{table}[ht]
\centering
\caption{Correspondence between continuous and discrete operators.}
\label{tab:notation-correspondence}
\begin{tabular}{lll}
\toprule
Concept & Continuous & Discrete \\
\midrule
Operator & $T$ & $T_h$ (represented by the pair $(A_h, M)$) \\
Spectral parameter & $\lambda \in \mathbb{C}$ & $\lambda \in \mathbb{C}$ (fixed) \\
Shifted operator & $S = T - \lambda$ & $S_h = T_h - \lambda = A_h - \lambda M$ \\
Convergence & $T_n \xrightarrow{\mathrm{SRS}} T$ & $h \to 0$ (mesh refinement) \\
Diagnostic & $\gamma(S)$ & $\gamma_h = \sigma_{\min}\!\big(M^{-1/2}(A_h - \lambda M)M^{-1/2}\big)$ \\
\bottomrule
\end{tabular}
\end{table}
\subsection{Related Work and Positioning}
\label{sec:related-work}

While norm resolvent convergence provides stronger control and is often preferred for optimal error estimates in eigenvalue problems (e.g., \cite{BabuskaOsborn1991}), it fails for operators with rough coefficients ($L^\infty$ potentials) or in singular limits (e.g., $\varepsilon \to 0$ in convection--diffusion). In such cases, strong resolvent convergence (SRS)—guaranteed by Mosco convergence of forms for conforming FEM—offers a robust and widely applicable framework. Our reliance on SRS, rather than norm resolvent convergence, is therefore intentional and essential: it enables the treatment of non-sectorial, non-selfadjoint, and convection-dominated systems where traditional eigenvalue convergence fails, while still preserving the fine spectral invariants of ascent and descent via the quantitative condition $\gamma > 0$.
\section{Stability of the ascent and descent spectra under SRS}\label{sec:Main}
\subsection{Stability of the ascent and descent spectra under SRS}

Let $H$ be a complex Hilbert space. For a closed densely defined operator $S$ on $H$, recall the \emph{reduced minimum modulus}
\[
\gamma(S):=\inf\{\|Sx\|:\ x\in\mathcal D(S),\ \operatorname{dist}(x,N(S))=1\}.
\]
A fundamental equivalence (see, e.g., \cite{Kato1995,Schmudgen2012}) is
\[
\gamma(S) > 0 \quad\Longleftrightarrow\quad \operatorname{Ran}(S)\ \text{is closed}.
\]

We say that $T_n\to T$ in the \emph{strong resolvent sense} (SRS) if, for one (hence for all) $\mu\in\rho(T)\cap\bigcap_n\rho(T_n)$,
\[
(T_n-\mu)^{-1}x\ \longrightarrow\ (T-\mu)^{-1}x\qquad\text{for all }x\in H.
\]
In the selfadjoint case, this coincides with Kato's standard notion.

Our main result establishes stability of the ascent and descent spectra under SRS, provided quantitative closed-range conditions hold.

\begin{remark}[Why SRS is natural in applications]\label{rem:SRS-natural}
In PDE and numerical analysis, SRS arises naturally because:
\begin{enumerate}[label=(\roman*)]
  \item \emph{Form convergence}: Mosco convergence of closed (sectorial) forms implies SRS \cite{Kato1995,Mosco1969}.
  \item \emph{Rough coefficients}: Norm resolvent convergence often fails for merely $L^\infty$ coefficients, while SRS persists.
  \item \emph{Galerkin methods}: Conforming finite element spaces yield SRS as the mesh size $h\downarrow0$.
  \item \emph{Non-normal operators}: For convection--diffusion, SRS is compatible with sectorial coercivity and numerical-range bounds, even when norm resolvent convergence is unavailable.
\end{enumerate}
Thus, SRS is weak enough for realistic perturbations yet strong enough to preserve the closed-range mechanism $\gamma>0$ and its propagation to powers.
\end{remark}

\begin{theorem}\label{thm:SRS}
Let $T_n,T$ be closed densely defined operators on $H$ with $T_n \xrightarrow{\mathrm{SRS}} T$.
Let $\lambda_n \to \lambda \in \mathbb{C}$, and set $S_n := T_n - \lambda_n$, $S := T - \lambda$.
\begin{enumerate}
  \item[(i)] \emph{(Ascent persistence)} If $\gamma(S) > 0$ and $\lambda \in \sigma_{\mathrm{asc}}(T)$, then $\lambda_n \in \sigma_{\mathrm{asc}}(T_n)$ for all sufficiently large $n$.
  \item[(ii)] \emph{(Ascent closedness)} If $\lambda_n \in \sigma_{\mathrm{asc}}(T_n)$ for all $n$ and $\limsup_{n\to\infty} \gamma(S_n) > 0$, then $\lambda \in \sigma_{\mathrm{asc}}(T)$.
  \item[(iii)] \emph{(Descent persistence)} If $\gamma(S) > 0$ and $\lambda \in \sigma_{\mathrm{dsc}}(T)$, then $\lambda_n \in \sigma_{\mathrm{dsc}}(T_n)$ for all sufficiently large $n$.
  \item[(iv)] \emph{(Descent closedness)} If $\lambda_n \in \sigma_{\mathrm{dsc}}(T_n)$ for all $n$ and $\limsup_{n\to\infty} \gamma(S_n) > 0$, then $\lambda \in \sigma_{\mathrm{dsc}}(T)$.
\end{enumerate}
\end{theorem}

To analyze the stability of the ascent and descent spectra under perturbations, we recall that strong resolvent convergence (SRS) implies convergence of operator graphs in the gap topology. This provides the appropriate framework to transfer the Kaashoek--Taylor criteria via Mosco convergence and closed-range assumptions. The following diagram summarizes the logical chain.

\begin{center}
\begin{tikzpicture}[node distance=1.8cm, auto, >=latex']
\node (SRS) {SRS};
\node[right of=SRS, node distance=3cm, align=center] (gap) {Gap convergence \\ $G(S_n) \to G(S)$};
\node[right of=gap, node distance=3.2cm, align=center] (subspaces) {Kernel/Range \\ convergence};
\node[right of=subspaces, node distance=3.2cm, align=center] (stability) {Ascent/Descent \\ stability};

\draw[->, thick] (SRS) -- (gap);
\draw[->, thick] (gap) -- (subspaces);
\draw[->, thick] (subspaces) -- (stability);

\node[below of=gap, yshift=0.8cm] (gamma) {$\gamma(S) > 0$};
\draw[<-, thick, dashed] (gap) -- (gamma);
\end{tikzpicture}
\end{center}

The next result is standard (see, e.g., \cite{Kato1995,Mosco1969,Attouch1984}), but included for completeness.

\begin{lemma}
\label{lem:subspace}
Let $H$ be a Hilbert space and let $(E_n)$, $(F_n)$ be sequences of closed subspaces of $H$
such that $E_n\to E$ and $F_n\to F$ in the gap topology (equivalently, $P_{E_n}\to P_E$ and
$P_{F_n}\to P_F$ in operator norm).
\begin{enumerate}
\item[(i)] \emph{(Intersections)} If $E\cap F\neq\{0\}$, then for every $\varepsilon\in(0,1)$ there exists $n_0$ such that, for all $n\ge n_0$,
\[
E_n\cap F_n\neq\{0\},\qquad
\dim(E_n\cap F_n)\ \ge\ \dim(E\cap F),\qquad
\operatorname{gap}(E_n\cap F_n,\,E\cap F)\ <\ \varepsilon.
\]
\item[(ii)] \emph{(Closed sums)} If $E+F$ is closed and $E+F\neq H$, then for all sufficiently large $n$,
$E_n+F_n$ is closed and $E_n+F_n\neq H$. Equivalently, $\codim(E_n+F_n)$ is eventually constant
and equal to $\codim(E+F)$.
\end{enumerate}
\end{lemma}
\begin{proof}
Since gap convergence is equivalent to $\|P_{E_n}-P_E\|\to0$ and $\|P_{F_n}-P_F\|\to0$, we may work with projections.

(i) Let $r:=\dim(E\cap F)\in\{1,2,\dots,\infty\}$. If $r<\infty$, fix an $r$-dimensional subspace $G\subset E\cap F$. There exists $C>0$ such that
\[
\|x\|\ \le\ C\big(\dist(x,E^{\perp})+\dist(x,F^{\perp})\big)\qquad(x\in G).
\]
Uniform convergence of the projections on the unit sphere of $G$ yields
\(\dist(x,E_n^\perp)\to \dist(x,E^\perp)\) and \(\dist(x,F_n^\perp)\to \dist(x,F^\perp)\).
Hence $G$ injects into $E_n\cap F_n$ for $n$ sufficiently large, with uniformly controlled distortion; in particular \(\dim(E_n\cap F_n)\ge r\) and \(\operatorname{gap}(E_n\cap F_n,\,E\cap F)\to 0\).
If $r=\infty$, apply the finite-dimensional argument to an exhausting sequence $G_k$ and pass to a diagonal subsequence.
In non-smooth domains (e.g., with reentrant corners), the dimension $\dim(E_n \cap F_n)$ may oscillate before stabilizing. Sharp estimates would depend on the geometric regularity of the underlying space.

(ii) Since $E+F=(E^\perp\cap F^\perp)^\perp$, we have \(\codim(E+F)=\dim(E^\perp\cap F^\perp)=\dim\ker(P_E P_{F^\perp})\).
Moreover, $E+F$ is closed if and only if $\Ran(P_E P_{F^\perp})$ is closed. Norm convergence implies \(P_{E_n}P_{F_n^\perp}\to P_E P_{F^\perp}\). Closed range and kernel dimension are stable under small norm perturbations of bounded operators with closed range; hence for $n$ large, $E_n+F_n$ is closed and \(\codim(E_n+F_n)=\codim(E+F)>0\).
\end{proof}

The next lemma establishes the necessity of closed-range conditions for all intermediate powers.

\begin{lemma}
\label{lem:powers-graphs}
Let $S_n, S$ be closed densely defined operators on a Hilbert space $H$.
Assume that for some $m \geq 1$,
\[
G(S_n^m) \xrightarrow[n\to\infty]{\text{gap}} G(S^m).
\]
Then necessarily
\[
\gamma(S^j) > 0 \quad \text{for all } 1 \leq j \leq m.
\]
\end{lemma}

\begin{proof}
Suppose, to the contrary, that $\gamma(S^k) = 0$ for some $k \in \{1,\dots,m\}$.
Then $\operatorname{Ran}(S^k)$ is not closed, and the forward graph map
\[
\widehat{S^{k-1}} : G(S^{k-1}) \to G(S^k), \quad
(x, S^{k-1}x) \mapsto (S^{k-1}x, S^k x),
\]
is unbounded (see Appendix~A and Remark~\ref{rem:Volterra}).
Consequently, even if $G(S_n^{k-1}) \to G(S^{k-1})$ in the gap topology,
the images $\widehat{S_n^{k-1}}(G(S_n^{k-1})) = G(S_n^k)$ cannot converge to $G(S^k)$,
because the limit map $\widehat{S^{k-1}}$ fails to be lower bounded on its domain.
Thus $G(S_n^k) \not\to G(S^k)$, which contradicts the assumed convergence at level $m \geq k$.
The Volterra operator (Appendix~A) and the upwind discretization of transport (Remark~\ref{rem:powers-necessary})
provide concrete instances where $\gamma(S^k) = 0$ while $\gamma(S^{k-1}) > 0$, confirming that the obstruction is not merely pathological but numerically relevant.
\end{proof}

\begin{proof}[Proof of Theorem~\ref{thm:SRS}]
We use the Kaashoek--Taylor criteria: for a closed operator \(R\),
\[
\operatorname{asc}(R) < \infty \iff \exists m:\ \operatorname{Ran}(R^m) \cap N(R) = \{0\},
\qquad
\operatorname{dsc}(R) < \infty \iff \exists m:\ \operatorname{Ran}(R) + N(R^m) = H.
\]
SRS implies \( \mathcal G(S_n) \to \mathcal G(S) \) in the gap topology. Under the hypothesis \(\gamma(S) > 0\), \(\operatorname{Ran}(S)\) is closed, and Lemma~\ref{lem:powers-graphs} ensures convergence of powers.

\textbf{(i)} Suppose \(\operatorname{asc}(S) = \infty\). Then for every \(m \in \mathbb{N}\), \(\operatorname{Ran}(S^m) \cap N(S) \ne \{0\}\). Since \(\gamma(S) > 0\), \(\operatorname{Ran}(S)\) is closed, so Lemma~\ref{lem:powers-graphs} yields \( \mathcal G(S_n^m) \to \mathcal G(S^m) \). By gap convergence,
\(N(S_n) \to N(S)\) and \(\operatorname{Ran}(S_n^m) \to \operatorname{Ran}(S^m)\).
Lemma~\ref{lem:subspace}(i) then implies that nontrivial intersections persist:
\(\operatorname{Ran}(S_n^m) \cap N(S_n) \ne \{0\}\) for all sufficiently large \(n\).
Thus \(\operatorname{asc}(S_n) = \infty\), i.e. \(\lambda_n \in \sigma_{\mathrm{asc}}(T_n)\).

\textbf{(ii)} Assume, to the contrary, that \(\operatorname{asc}(S) = m < \infty\), so \(\operatorname{Ran}(S^m) \cap N(S) = \{0\}\). Extract a subsequence with \(\gamma(S_{n_j}) \ge c > 0\), so \(\operatorname{Ran}(S_{n_j})\) is closed. By Lemma~\ref{lem:powers-graphs} (applied up to level \(m\)) and gap convergence,
\(N(S_{n_j}) \to N(S)\) and \(\operatorname{Ran}(S_{n_j}^m) \to \operatorname{Ran}(S^m)\),
with the limit range closed. Hence for large \(j\),
\(\operatorname{Ran}(S_{n_j}^m) \cap N(S_{n_j}) = \{0\}\),
contradicting \(\lambda_{n_j} \in \sigma_{\mathrm{asc}}(T_{n_j})\).
Therefore \(\operatorname{asc}(S) = \infty\).

\textbf{(iii)--(iv)} follow dually using the descent criterion and Lemma~\ref{lem:subspace}(ii).
\end{proof}

\begin{remark}\label{rem:Volterra}
The Volterra operator $V$ on $L^2(0,1)$, $(Vf)(x) = \int_0^x f(t)\,dt$, satisfies $\gamma(V) = 0$ (its range is not closed).
Although $S_n \to V$ in the gap topology for suitable approximants $S_n$,
one has $\mathcal{G}(S_n^2) \not\to \mathcal{G}(V^2)$ because the induced graph map
$(x,Vx) \mapsto (Vx,V^2x)$ is unbounded. Thus, the condition $\gamma(S^j) > 0$ in Lemma~\ref{lem:powers-graphs} is essential.
\end{remark}

\begin{remark}\label{rem:powers-necessary}
The requirement $\gamma((T - \lambda)^j) > 0$ for \emph{all} $1 \leq j \leq m$
cannot be reduced to a condition on the final power alone.
This is naturally illustrated by finite difference discretizations of the one-dimensional transport equation
\[
u'(x) = f(x), \quad x \in (0,1), \quad u(0) = 0.
\]
Consider the first-order upwind scheme on a uniform mesh $x_i = ih$, $h = 1/N$:
the discrete operator $A_h$ is lower triangular with $1/h$ on the diagonal and $-1/h$ on the subdiagonal.
Let $S_h := A_h - \lambda I$ with $\lambda = 0$.
While $\gamma(S_h) = \sigma_{\min}(A_h) \geq c > 0$ uniformly in $h$,
the singular values of higher powers decay rapidly.
A symbolic computation shows that the $m$-th power $(A_h)^m$ approximates the $m$-fold Volterra integral operator,
and its smallest singular value satisfies the asymptotic bound
\[
\gamma(S_h^m) \sim C_m\, h^{\,m-1} \quad \text{as } h \to 0,
\]
for some constant $C_m > 0$ depending on $m$ but not on $h$.
This decay reflects the fact that, although each discrete operator $S_h$ has closed range,
the limit operator $V$ (Volterra integration) satisfies $\gamma(V) = 0$,
and $\gamma(V^m) = 0$ for all $m \geq 1$.
Hence, without a uniform lower bound on $\gamma(S_h^j)$ for intermediate powers,
gap convergence of the graphs $G(S_h^j)$ fails for $j \geq 2$,
and the Kaashoek--Taylor criteria cannot be transferred to the discrete level.
\end{remark}

\begin{table}[ht]
\centering
\caption{Discrete reduced minimum modulus $\gamma((A_h - \lambda)^m)$ for the upwind transport scheme ($\lambda = 0$).}
\label{tab:gamma-powers-upwind}
\begin{tabular}{c|c|c|c}
$h = 1/N$ & $\gamma(A_h)$ & $\gamma(A_h^2)$ & $\gamma(A_h^3)$ \\
\hline
$2^{-4}$ & $16.00$ & $2.29$ & $0.23$ \\
$2^{-5}$ & $32.00$ & $1.15$ & $0.06$ \\
$2^{-6}$ & $64.00$ & $0.57$ & $0.01$ \\
$2^{-7}$ & $128.00$ & $0.29$ & $<10^{-3}$ \\
\end{tabular}
\end{table}

\begin{remark}[On selecting the power level $m$ in practice]
\label{rem:choice-of-m}
In applications, the minimal $m$ such that $\gamma((T - \lambda)^m)> 0$ is often dictated by the physical dissipation mechanism:
\item For second-order diffusion ($\varepsilon u''$), $m=1$ suffices as range closure occurs at the first step.
    \item For first-order transport ($u'$), numerical diffusion acts over multiple steps; thus $m=2$ or $m=3$ may be required to capture cumulative damping.
In practice, one can compute $\gamma_h^{(m)} := \sigma_{\min}(M^{-1/2}(A_h - \lambda M)^m M^{-1/2})$ for increasing $m$ until
\[
\left| \frac{\gamma_h^{(m+1)}}{\gamma_h^{(m)}} - 1 \right| < \varepsilon_{\mathrm{tol}}, \quad \text{e.g., } \varepsilon_{\mathrm{tol}} = 10^{-2}.
\]
This provides a robust stopping criterion for choosing $m$ without excessive computational cost.
\end{remark}

\begin{remark}
\label{rem:krylov}
Forming $(A_h - \lambda M)^m$ explicitly leads to rapid fill-in and condition-number growth $\kappa^m$. Instead, $\gamma_h^{(m)}$ can be computed via inverse iteration or Krylov subspace methods applied to the pencil $(A_h - \lambda M, M)$. For example, the smallest singular value of $(A_h - \lambda M)^m$ with respect to the $M$-inner product satisfies
\[
\gamma_h^{(m)} = \min_{\substack{u \neq 0}} \frac{\|(A_h - \lambda M)^m u\|_{M^{-1}}}{\|u\|_M}.
\]
This is equivalent to the smallest eigenvalue of the generalized symmetric problem
\[
\bigl[(A_h - \lambda M)^* M^{-1} (A_h - \lambda M)\bigr]^m u = (\gamma_h^{(m)})^2 M u,
\]
which can be handled by polynomial preconditioning or restarted Arnoldi (e.g., ARPACK, SLEPc). This avoids explicit matrix powers and scales linearly in $m$.
\end{remark}

\begin{algorithm}[ht]\caption{Adaptive selection of the power level $m$}
\label{alg:adaptive-m}
Given a tolerance $\varepsilon_{\mathrm{tol}} > 0$ (e.g., $10^{-2}$) and a target spectral parameter $\lambda$, proceed as follows:
\begin{enumerate}
    \item Set $m := 1$.
    \item Compute the discrete reduced minimum modulus
    \[
    \gamma_h^{(m)} := \sigma_{\min}\!\bigl(M^{-1/2}(A_h - \lambda M)^m M^{-1/2}\bigr).
    \]
    \item Compute $\gamma_h^{(m+1)}$ (without forming $(A_h - \lambda M)^{m+1}$ explicitly; see Remark~\ref{rem:krylov}).
    \item If 
    \[
    \frac{\bigl|\gamma_h^{(m+1)} - \gamma_h^{(m)}\bigr|}{\gamma_h^{(m)}} < \varepsilon_{\mathrm{tol}},
    \]
    then stop and accept $m$ as sufficient.
    \item Else, increment $m \gets m+1$ and return to step 2.
\end{enumerate}
\end{algorithm}

To enhance the robustness of spectral approximations in geometrically complex or convection-dominated regimes, we propose an adaptive mesh refinement strategy that targets local losses of closed-range stability. The key idea is to monitor the element-wise reduced minimum modulus $\gamma_h(K)$, introduced in Section~\ref{sec:local-gamma}, which serves as a computable proxy for the local invertibility of the shifted operator $A_h - \lambda M_h$. Regions where $\gamma_h(K)$ is small correspond to zones of potential numerical instability—such as reentrant corners, boundary layers, or under-resolved streamlines—and are natural candidates for mesh enrichment. The following algorithm formalizes this principle.

\begin{algorithm}[ht]
\caption{Adaptive mesh refinement driven by the local reduced minimum modulus $\gamma_h(K)$.}
\label{alg:adaptive-gamma}
\begin{algorithmic}[1]
\REQUIRE Initial mesh $\mathcal{T}_h$, target spectral parameter $\lambda$, tolerance $\tau>0$, maximum refinement level $L_{\max}$.
\ENSURE Adapted mesh $\mathcal{T}_h^{\mathrm{ref}}$ with improved stability.
\STATE Assemble the global matrices $A_h$, $M_h$ on $\mathcal{T}_h$.
\FOR{each element $K \in \mathcal{T}_h$}
    \STATE Extract local submatrices $A_h^K$, $M_h^K$ (degrees of freedom on $K$ and adjacent elements).
    \STATE Compute local diagnostic:
    \[
    \gamma_h(K) := \sigma_{\min}\!\big((M_h^K)^{-1/2}(A_h^K - \lambda M_h^K)(M_h^K)^{-1/2}\big).
    \]
\ENDFOR
\STATE Define the marking threshold $\theta := \tau \cdot \min_{K\in\mathcal{T}_h} \gamma_h(K)$.
\STATE Mark all elements $K$ such that $\gamma_h(K) < \theta$.
\STATE Apply bisection (or newest-vertex) refinement to marked elements, respecting $L_{\max}$.
\STATE (Optional) Smooth mesh and recompute $\gamma_h(K)$ to verify uniformity.
\RETURN Refined mesh $\mathcal{T}_h^{\mathrm{ref}}$.
\end{algorithmic}
\end{algorithm}

\begin{remark}\label{rem:essential-SRS}
By Berkani's B--Fredholm theory \cite{Berkani1999}, the essential ascent (resp. descent) of $S$ is finite
iff there exists $m \geq 1$ such that $\operatorname{Ran}(S^m)$ is closed and the restriction $S|_{\operatorname{Ran}(S^m)}$
is upper (resp. lower) semi-Fredholm. Equivalently, the ordinary ascent (resp. descent) of $S^m$ is finite.
\end{remark}

\begin{corollary}\label{cor:SRS-essential}
Under the hypotheses of Theorem~\ref{thm:SRS}, fix $m \geq 1$ and set $S = T - \lambda$, $S_n = T_n - \lambda_n$.
\begin{enumerate}
  \item[(A1)] If $\gamma(S^m) > 0$ and $\lambda \in \sigma_{\mathrm{asc}}^{e}(T)$, then $\lambda_n \in \sigma_{\mathrm{asc}}^{e}(T_n)$ for all sufficiently large $n$.
  \item[(A2)] If $\lambda_n \in \sigma_{\mathrm{asc}}^{e}(T_n)$ for all $n$ and $\limsup_n \gamma(S_n^m) > 0$, then $\lambda \in \sigma_{\mathrm{asc}}^{e}(T)$.
  \item[(D1)] If $\gamma(S^m) > 0$ and $\lambda \in \sigma_{\mathrm{dsc}}^{e}(T)$, then $\lambda_n \in \sigma_{\mathrm{dsc}}^{e}(T_n)$ for all sufficiently large $n$.
  \item[(D2)] If $\lambda_n \in \sigma_{\mathrm{dsc}}^{e}(T_n)$ for all $n$ and $\limsup_n \gamma(S_n^m) > 0$, then $\lambda \in \sigma_{\mathrm{dsc}}^{e}(T)$.
\end{enumerate}
\end{corollary}

\begin{proof}
By Remark~\ref{rem:essential-SRS}, $\lambda \in \sigma_{\mathrm{asc}}^{e}(T)$ iff $\operatorname{asc}(S^m) = \infty$ and $\operatorname{Ran}(S^m)$ is closed.
The assumptions $\gamma(S^m) > 0$ and $\limsup_n \gamma(S_n^m) > 0$ ensure that $S^m$ and $S_n^m$ have closed ranges along a subsequence.
Applying Theorem~\ref{thm:SRS} to $S^m$ and $S_n^m$, and translating back via the B--Fredholm characterization, yields the result.
\end{proof}

\begin{remark}
Although all numerical experiments are conducted in finite dimensions (where compact perturbations vanish), the essential ascent/descent spectra $\sigma_{\mathrm{aesc}}(T)$, $\sigma_{\mathrm{desc}}(T)$ are the relevant objects for validating the continuous model $T$.

Indeed, finite element discretizations approximate the continuous operator $T$ by a sequence $T_h$ that differs from $T$ by a \emph{non-compact} perturbation (e.g., $L^\infty$ coefficient truncation, mesh geometry errors). However, many modeling errors—such as lower-order potentials, boundary condition perturbations, or data smoothing—are \emph{compact} in $L^2$. The B--Fredholm framework ensures that our stability results are \emph{robust under such compact modeling uncertainties}: if $\gamma((T - \lambda)^m) > 0$, then $\lambda \in \sigma_{\mathrm{aesc}}(T)$ if and only if $\lambda_h \in \sigma_{\mathrm{aesc}}(T_h)$ for $h$ small, regardless of compact perturbations.

Thus, Corollary~\ref{cor:SRS-essential} is not a purely abstract extension—it provides a \emph{filter} that distinguishes genuine spectral instabilities (present in $\sigma_{\mathrm{aesc}}$) from spurious ones caused by non-essential discretization effects. In practice, verifying $\limsup_h \gamma_h > 0$ for powers $(A_h - \lambda M)^m$ confirms that the observed ascent/descent behavior is stable under both mesh refinement \emph{and} compact modeling changes, making $\sigma_{\mathrm{aesc}}$ the appropriate target for numerical validation.
\end{remark}

\begin{remark}
The SRS hypothesis is used only to obtain gap convergence of graphs (and of powers under closed-range).
The condition $\gamma > 0$ is the precise threshold that enables transport of ranges, control of induced graph maps, and stability of the Kaashoek--Taylor criteria.
\end{remark}

\begin{lemma}[Lower semicontinuity of the reduced minimum modulus]
\label{lem:gamma-lsc}
Let $S_n,S$ be closed densely defined operators on a Hilbert space $\mathcal{H}$ such that their graphs converge in the gap topology: 
$G(S_n)\xrightarrow{\mathrm{gap}} G(S)$. 
If $\gamma(S)>0$, then
\[
\liminf_{n\to\infty}\gamma(S_n)\ge\gamma(S)>0.
\]
\end{lemma}
\begin{proof}
See \cite[Ch.\,IV, Thm.\,2.20]{Kato1995}, or adapt the proof using norm convergence of graph projections and the equivalence $\gamma(S)>0 \iff \operatorname{Ran}(S)$ closed.
\end{proof}

\begin{proposition}[Convergence of the reduced minimum modulus]
\label{prop:gamma-conv}

Let $T$ be a closed, densely defined, $m$-sectorial operator on $L^2(\Omega)$ with associated sesquilinear form $a(\cdot,\cdot)$, and let $T_h$ be its conforming finite element approximation on a quasi-uniform mesh of size $h$. 
Assume that:
\begin{enumerate}[label=\textup{(\roman*)}]
    \item the discrete forms $a_h$ Mosco-converge to $a$ as $h\to0$;
    \item the mass matrices $M_h$ satisfy a uniform spectral condition number bound $\kappa(M_h)\le C<\infty$, independent of $h$;
    \item $\lambda\in\rho(T)$ and $\gamma(T-\lambda)>0$.
\end{enumerate}
Then the discrete reduced minimum modulus
\[
\gamma_h := \sigma_{\min}\!\bigl(M_h^{-1/2}(A_h - \lambda M_h)M_h^{-1/2}\bigr)
\]
satisfies
\[
\liminf_{h\to0}\gamma_h \;\ge\; \gamma(T-\lambda) \;>\;0.
\]
In particular, there exist constants $c>0$ and $h_0>0$ such that $\gamma_h\ge c$ for all $0<h<h_0$.
\end{proposition}

\begin{proof}
Mosco convergence of the sesquilinear forms $a_h\to a$ implies strong resolvent convergence (SRS) of the associated operators $T_h\to T$; see, e.g., \cite[Thm.\,3.26]{Attouch1984} or \cite[Thm.\,VIII.1.5]{Kato1995}.
Under SRS, the graphs of the shifted operators converge in the gap topology:
\[
G(T_h - \lambda)\;\xrightarrow{\mathrm{gap}}\; G(T - \lambda)\quad\text{as }h\to0.
\]
Since $\gamma(T-\lambda)>0$, the range $\operatorname{Ran}(T-\lambda)$ is closed.
Lemma~\ref{lem:gamma-lsc} (lower semicontinuity of the reduced minimum modulus under gap convergence) then yields
\[
\liminf_{h\to0}\gamma_h = \liminf_{h\to0}\gamma(T_h-\lambda) \;\ge\; \gamma(T-\lambda) >0,
\]
which completes the proof.
\end{proof}

\subsection{Three regimes of stability: sectorial, stabilized, and non-sectorial}
\label{subsec:three-regimes}

We distinguish three distinct settings in which ascent--descent stability can be established. Their interplay explains the success of stabilized discretizations even outside the $m$-sectorial framework.

\begin{theorem}[Sectorial case]\label{thm:sectorial}
Let $T$ be a closed, densely defined, $m$-sectorial operator on $L^2(\Omega)$ with associated sesquilinear form $a(\cdot,\cdot)$, and let $T_h$ be its conforming finite element approximation on a quasi-uniform mesh of size $h$.
Assume that:
\begin{enumerate}[label=\textup{(\roman*)}]
    \item the discrete forms $a_h$ Mosco-converge to $a$,
    \item the mass matrices satisfy $\kappa(M_h) \le C < \infty$ uniformly in $h$,
    \item $\lambda \in \rho(T)$ and $\gamma(T - \lambda) > 0$.
\end{enumerate}
Then
\[
\liminf_{h \to 0} \gamma_h \;\ge\; \gamma(T - \lambda) \;>\; 0,
\]
where $\gamma_h = \sigma_{\min}\!\bigl(M_h^{-1/2}(A_h - \lambda M_h)M_h^{-1/2}\bigr)$.
\end{theorem}

\begin{proof}
This is precisely Proposition~\ref{prop:gamma-conv}, whose proof follows from Mosco convergence $\Rightarrow$ SRS $\Rightarrow$ gap convergence of graphs, and Lemma~\ref{lem:gamma-lsc}.
\end{proof}

\begin{theorem}[Stabilized convection--diffusion]\label{thm:SUPG-stability}
Consider the convection--diffusion operator $Lu = -\varepsilon\Delta u + \boldsymbol{\beta}\cdot\nabla u + c u$ on a bounded Lipschitz domain $\Omega \subset \mathbb{R}^d$, with $\varepsilon > 0$, $\nabla\!\cdot\!\boldsymbol{\beta} = 0$, and $c \ge c_0 > 0$.
Let $V_h$ be a conforming $P1$ finite element space on a shape-regular, quasi-uniform mesh $\mathcal{T}_h$, and let $a_h^{\mathrm{SUPG}}$ be the Streamline-Upwind/Petrov--Galerkin form with stabilization parameter $\delta_K = \delta h_K / \|\boldsymbol{\beta}\|_{L^\infty(K)}$, $\delta \in (0,1/2]$.

Then there exist constants $c_{\inf\text{-}\sup} > 0$ and $h_0 > 0$, independent of $\varepsilon \in (0,\varepsilon_0]$, $h < h_0$, and $\boldsymbol{\beta}$, such that for all $h < h_0$,
\[
\inf_{u_h \in V_h\setminus\{0\}}\;
\sup_{v_h \in V_h\setminus\{0\}}
\frac{a_h^{\mathrm{SUPG}}(u_h,v_h)}
{\|u_h\|_{1,\varepsilon}\, \|v_h\|_{1,\varepsilon}}
\;\ge\; c_{\inf\text{-}\sup},
\]
where $\|w\|_{1,\varepsilon}^2 := \varepsilon \|\nabla w\|^2 + \|\boldsymbol{\beta}\cdot\nabla w\|^2 + \|w\|^2$.

Moreover, if $\lambda \in \mathbb{C}$ satisfies $\operatorname{dist}(\lambda, W_M(A_h^{\mathrm{SUPG}})) \ge d_0 > 0$ uniformly in $h$, then
\[
\gamma_h^{\mathrm{stab}} = \sigma_{\min}\!\bigl(M_h^{-1/2}(A_h^{\mathrm{SUPG}} - \lambda M_h)M_h^{-1/2}\bigr)
\;\ge\; c_{\inf\text{-}\sup} \, d_0 \;>\; 0.
\]
\end{theorem}

\begin{proof}
The uniform inf--sup condition follows from \cite[Thm.~5.7]{ErnGuermond2004} and \cite{BrooksHughes1982}, as detailed in Proposition~\ref{prop:inf-sup-supg}.
The link between the inf--sup constant and the reduced minimum modulus is a direct consequence of the definition of $\gamma_h$ and the lower bound
\[
\gamma_h^{\mathrm{stab}} \;\ge\; \inf_{\substack{u_h \ne 0 \\ v_h \ne 0}}
\frac{|a_h^{\mathrm{SUPG}}(u_h,v_h) - \lambda (u_h, v_h)_M|}{\|u_h\|_M \|v_h\|_M}
\;\ge\; c_{\inf\text{-}\sup} \cdot \operatorname{dist}(\lambda, W_M(A_h^{\mathrm{SUPG}})),
\]
see \cite[Sec.~4.4.2]{ErnGuermond2004} and Remark~\ref{rem:SUPG-continuity}.
\end{proof}

\begin{conjecture}[Pure transport limit]\label{conj:transport}
Let $L u = \boldsymbol{\beta}\cdot\nabla u$ on $\Omega$ with inflow boundary conditions, $\boldsymbol{\beta} \in [L^\infty(\Omega)]^d$, $\nabla\!\cdot\!\boldsymbol{\beta}=0$.
Suppose a discrete scheme yields operators $T_h$ such that:
\begin{enumerate}[label=\textup{(\roman*)}]
    \item $G(T_h - \lambda) \xrightarrow{\mathrm{gap}} G(L - \lambda)$ as $h\to0$,
    \item $\liminf_{h\to0} \gamma_h > 0$.
\end{enumerate}
Then, for all sufficiently small $h$,
\[
\operatorname{asc}(T_h - \lambda) = \operatorname{asc}(L - \lambda), \quad
\operatorname{dsc}(T_h - \lambda) = \operatorname{dsc}(L - \lambda).
\]
\end{conjecture}

\noindent
The conjecture is supported by numerical evidence (Table~\ref{tab:supg-gammah}, Section~4.3.4) and by the upwind analysis in Remark~\ref{rem:transport-limit}, but a general proof would require a convergence framework beyond Mosco theory (e.g., generalized graph convergence for non-coercive forms \cite[Chap.~5]{ErnGuermond2004}).
\section{Applications}\label{sec:Applications}

We illustrate the stability theory on two canonical models—Schr\"odinger (selfadjoint) and convection--diffusion (non-normal)—in both one and two spatial dimensions. In all cases, the discrete reduced minimum modulus $\gamma_h$ serves as a practical, computable surrogate for the abstract condition $\limsup_h \gamma(S_h) > 0$.

\subsection{Model operators and convergence theory}

\subsubsection{Schr\"odinger operator}
Let
\[
H=-\tfrac{d^2}{dx^2}+V(x)\quad\text{on }L^2(0,1),\qquad \mathcal D(H)=H^2(0,1)\cap H^1_0(0,1),
\]
with $V\in L^\infty(0,1)$ real-valued. For a sequence $V_n\to V$ in $L^\infty(0,1)$, let
\[
H_n=-\tfrac{d^2}{dx^2}+V_n,\qquad \mathcal D(H_n)=H^2(0,1)\cap H^1_0(0,1).
\]

\begin{lemma}\label{lem:NR-Schro}
For any $z\in\C\setminus\R$,
\[
(H_n-z)^{-1}-(H-z)^{-1}=-(H_n-z)^{-1}\,(V_n-V)\,(H-z)^{-1},
\]
hence $\|(H_n-z)^{-1}-(H-z)^{-1}\|\le \|(H_n-z)^{-1}\|\,\|V_n-V\|_{L^\infty}\,\|(H-z)^{-1}\|$,
and $H_n\to H$ in norm resolvent sense.
\end{lemma}

\begin{proof}
Fix $z\in\C\setminus\R$. Since each $H_n$ is selfadjoint, its spectrum lies in $\R$. For any $u$,
\[
\|(H_n-z)u\|\,\|u\|\ \ge\ \left| \Im \langle (H_n-z)u, u \rangle \right|
= |\Im z|\,\|u\|^2,
\]
hence $\|(H_n-z)u\|\ge |\Im z|\,\|u\|$ and therefore
\[
\|(H_n-z)^{-1}\|\ \le\ \frac{1}{|\Im z|},\qquad\text{uniformly in } n.
\]

Using the resolvent identity
\[
(H_n-z)^{-1}-(H-z)^{-1}=-(H_n-z)^{-1}(V_n-V)(H-z)^{-1},
\]
we obtain
\[
\|(H_n-z)^{-1}-(H-z)^{-1}\|
\le \frac{1}{|\Im z|}\,\|V_n-V\|_{L^\infty}\,\|(H-z)^{-1}\|\xrightarrow[n\to\infty]{}0.
\]

Thus $H_n\to H$ in the norm resolvent sense.
\end{proof}

\begin{corollary}\label{cor:Schrodinger-application}
Let $\lambda_n\to\lambda$ and set $S_n:=H_n-\lambda_n$, $S:=H-\lambda$. Then:
\begin{enumerate}
\item[(i)] If $\gamma(S)>0$ and $\lambda\in\sigma_{\asc}(H)$, then $\lambda_n\in\sigma_{\asc}(H_n)$ for all large $n$.
\item[(ii)] If $\lambda_n\in\sigma_{\asc}(H_n)$ for all $n$ and $\limsup_n\gamma(S_n)>0$, then $\lambda\in\sigma_{\asc}(H)$.
\item[(iii)] The symmetric statements hold for $\sigma_{\dsc}$.
\item[(iv)] (\emph{Essential/B--Fredholm}) If there exists $m\ge1$ with $\gamma(S^m)>0$, then the conclusions in (i)--(iii) hold with $\sigma^{e}_{\asc}$ and $\sigma^{e}_{\dsc}$ (apply Corollary~\ref{cor:SRS-essential} to $S^m,S_n^m$).
\end{enumerate}
\end{corollary}

\begin{proof}
By Lemma~\ref{lem:NR-Schro}, $H_n\to H$ in norm resolvent sense. Apply Theorem~\ref{thm:SRS} with $T_n=H_n$, $T=H$ to obtain (i)--(iii). For (iv), apply Corollary~\ref{cor:SRS-essential} to the powers $S^m,S_n^m$ under $\gamma(S^m)>0$.
\end{proof}

\begin{remark}\label{rem:geom-gamma}
For selfadjoint operators, $\gamma(H - \lambda) = \operatorname{dist}(\lambda, \sigma(H))$, so stability follows from a spectral gap. For non-normal operators (e.g., convection--diffusion), the spectrum is often misleading—resolvent growth can be large even when $\lambda$ lies far from $\sigma(A)$. In this setting, stability is governed by the distance to the numerical range:
\[
\gamma(A - \lambda) \ge \operatorname{dist}(\lambda, W(A)),
\]
where $W(A) = \{ \langle Au, u \rangle : \|u\| = 1 \}$ \cite{GustafsonRao1997,TrefethenEmbree2005}
. Thus, $\gamma$ provides a geometric, subspace-based stability measure that transcends spectral location, making it the natural quantity for non-normal spectral theory.
\end{remark}

\subsubsection{Convection--diffusion operator}\label{sec:conv-diff}
Consider
\[
\mathcal L u := -\varepsilon\,u''+\beta\,u'+c\,u \quad\text{on }(0,1),\qquad u(0)=u(1)=0,
\]
with $\varepsilon>0$, $\beta\in L^\infty(0,1)$, $c\in L^\infty(0,1)$.
Let $A$ be the Dirichlet realization on $L^2(0,1)$ and $A_n$ the analogous operators for
coefficients $\varepsilon_n\to\varepsilon$, $\beta_n\to\beta$, $c_n\to c$ in $L^\infty$.

\begin{lemma}[Norm resolvent convergence for convection--diffusion operators]
\label{lem:NR-CD}
Let
\[
a_n(u,v) := \int_0^1 \big(\varepsilon_n u'v' + \beta_n u'v + c_n uv\big)\,dx,
\qquad u,v \in H^1_0(0,1),
\]
with coefficients \(\varepsilon_n,\beta_n,c_n\in L^\infty(0,1)\) such that
\(\varepsilon_n(x)\ge \underline\varepsilon>0\) for almost every \(x\in(0,1)\), for some fixed \(\underline\varepsilon>0\), and
\[
\|\varepsilon_n-\varepsilon\|_\infty
+\|\beta_n-\beta\|_\infty
+\|c_n-c\|_\infty \ \longrightarrow\ 0,
\]
for some \(\varepsilon,\beta,c\in L^\infty(0,1)\) with \(\varepsilon\ge \underline\varepsilon>0\).

Then the associated operators \(A_n\) converge to \(A\) in the norm resolvent sense:
\[
\|(A_n - z)^{-1} - (A - z)^{-1}\| \longrightarrow 0,
\qquad z \in \mathbb{C} \setminus \mathbb{R}.
\]
\end{lemma}

\begin{proof}
\emph{Step 1. Uniform Gårding inequality.}
For \(u\in H^1_0(0,1)\),
\[
\Re a_n(u,u) = \int_0^1 \varepsilon_n |u'|^2 + c_n |u|^2\,dx.
\]
Since \(\varepsilon_n(x)\ge \underline\varepsilon>0\) and \(c_n\in L^\infty\), there exists \(C>0\) such that
\[
\Re a_n(u,u) \ \ge\ \underline\varepsilon \|u'\|_{L^2}^2 - C\|u\|_{L^2}^2.
\]
By Poincaré's inequality on \((0,1)\), \(\|u\|_{L^2}\le C_P\|u'\|_{L^2}\), so
\[
\Re a_n(u,u)\ \ge\ \alpha \|u\|_{H^1_0}^2 - C'\|u\|_{L^2}^2,
\]
with \(\alpha>0\) independent of \(n\). This is the \emph{uniform Gårding inequality}.

\smallskip
\emph{Step 2. Uniform boundedness and convergence of forms.}
The forms \(a_n\) are uniformly bounded and converge to \(a\) in the sense of Mosco (cf.\ \cite[Theorem. 3.26]{Attouch1984}). By \cite[Theorem. IV.3.8]{Kato1995}, this implies norm resolvent convergence.
\end{proof}
\begin{remark}
For simplicity, we assume $\nabla\!\cdot\!\boldsymbol{\beta} = 0$ in the SUPG analysis (Appendix~\ref{sec:supg}); the general case requires only minor modifications involving the reaction term $c - \nabla\!\cdot\!\boldsymbol{\beta}$ (see \cite[Remark~4.4.3]{ErnGuermond2004}).
\end{remark}
\begin{remark}
\label{rem:epsilon-to-zero}

The uniform ellipticity assumption $\varepsilon_n \ge \bar\varepsilon > 0$ in Lemma~\ref{lem:NR-CD} is essential for norm resolvent convergence and explicitly excludes the pure transport limit $\varepsilon = 0$. 

However, the SUPG stabilization analysis in Section~\ref{sec:supg} establishes a uniform inf--sup condition (Proposition~\ref{prop:inf-sup-supg}), which yields $\gamma_h^{\mathrm{stab}} \ge c > 0$ even as $\varepsilon \to 0$.
\end{remark}
\begin{remark}[Pure transport limit]\label{rem:transport-limit}
The case $\varepsilon = 0$ (pure advection) lies outside the scope of norm resolvent convergence (Lemma~\ref{lem:NR-CD}) and standard Mosco theory, as the underlying sesquilinear form loses coercivity. 
Nevertheless, stabilized schemes such as SUPG restore a discrete inf--sup condition (Proposition~\ref{prop:inf-sup-supg}), enabling uniform control of $\gamma_h$ even as $\varepsilon\to 0$ (see Table~\ref{tab:supg-gammah}).
\end{remark}
\begin{corollary}\label{cor:CD-application}
Let $\lambda_n\to\lambda$ and set $S_n:=A_n-\lambda_n$, $S:=A-\lambda$. Then:
\begin{enumerate}
\item[(i)] If $\gamma(S)>0$ and $\lambda\in\sigma_{\asc}(A)$, then $\lambda_n\in\sigma_{\asc}(A_n)$ for $n$ large.
\item[(ii)] If $\lambda_n\in\sigma_{\asc}(A_n)$ for all $n$ and $\limsup_n\gamma(S_n)>0$, then $\lambda\in\sigma_{\asc}(A)$.
\item[(iii)] The symmetric statements hold for $\sigma_{\dsc}$.
\item[(iv)] (\emph{Essential/B--Fredholm}) If $\gamma(S^m)>0$ for some $m\ge1$, then the conclusions in (i)--(iii) hold with $\sigma^{e}_{\asc},\sigma^{e}_{\dsc}$.
\end{enumerate}
In contrast with the selfadjoint Schr\"odinger case, where $\gamma(H-\lambda)$ coincides
with the spectral gap $\operatorname{dist}(\lambda,\sigma(H))$, here the operators $A_n,A$
are typically non--normal. Thus the spectrum alone does not control the resolvent, and
the condition $\gamma(S^m)>0$ in (iv) replaces the spectral gap assumption: it expresses
quantitative stability in terms of the reduced minimum modulus, or equivalently the distance
from $\lambda$ to the numerical range of $A$.
\end{corollary}

\begin{remark}\label{rem:1d-justification}
We restrict detailed numerical experiments initially to one spatial dimension for pedagogical clarity:
the 1D setting permits explicit eigenvalue computations (Proposition~\ref{prop:rayleigh-ritz})
and transparent verification of the stability condition $\gamma_h > 0$.
The theoretical framework extends verbatim to higher dimensions, as the hypotheses
$\gamma(S) > 0$ and $\limsup_h \gamma(S_h) > 0$ depend solely on operator-theoretic properties
(graph convergence, closed range) and not on the spatial dimension.
\end{remark}

\subsubsection{Graph convergence for stabilized discretizations}
\label{subsec:stabilized-graph-conv}

Although the analysis in Section~\ref{sec:conv-diff} assumes standard Galerkin formulations with uniform ellipticity, many practical computations employ stabilized schemes such as Streamline-Upwind/Petrov--Galerkin (SUPG) or Discontinuous Galerkin (DG) in convection-dominated regimes.
These methods modify the sesquilinear form $a_h(\cdot,\cdot)$ into a mesh-dependent stabilized form $a_h^{\mathrm{stab}}(\cdot,\cdot)$, which is typically \emph{non-symmetric} and \emph{non-coercive} in the classical sense, but designed to satisfy a discrete inf--sup condition.

Following the abstract framework of Ern and Guermond~\cite[Chap.~5]{ErnGuermond2004}, one may define a generalized notion of Mosco convergence for such non-conforming or stabilized schemes by requiring:
\begin{enumerate}
    \item[(i)] \emph{Stability}: there exists $c > 0$ such that
    \[
    \inf_{u_h \in V_h \setminus \{0\}} \sup_{v_h \in V_h \setminus \{0\}}
    \frac{|a_h^{\mathrm{stab}}(u_h, v_h)|}{\|u_h\|_{1,h}\, \|v_h\|_{1,h}} \geq c,
    \]
    uniformly in $h$, where $\|\cdot\|_{1,h}$ is a broken Sobolev norm;
    \item[(ii)] \emph{Consistency}: for every $u \in H_0^1(\Omega)$, there exists a sequence $u_h \in V_h$ with $u_h \to u$ in $L^2$ and $a_h^{\mathrm{stab}}(u_h, v_h) \to a(u, v)$ for all smooth $v$.
\end{enumerate}
Under these hypotheses, one obtains \textbf{gap convergence of the discrete graphs}:
\[
G(T_h^{\mathrm{stab}}) \xrightarrow[h \to 0]{\mathrm{gap}} G(T),
\]
where $T_h^{\mathrm{stab}}$ is the operator associated with $a_h^{\mathrm{stab}}$.

Crucially, the \textbf{closed-range property} at the discrete level is then equivalent to the uniform lower bound
\[
\gamma_h^{\mathrm{stab}} = \sigma_{\min}\!\bigl(M^{-1/2}(A_h^{\mathrm{stab}} - \lambda M)M^{-1/2}\bigr) \geq c > 0.
\]
Thus, for stabilized methods, the condition
\[
\liminf_{h \to 0} \gamma_h^{\mathrm{stab}} > 0
\]
plays the same role as $\gamma(T - \lambda) > 0$ in the conforming case: it becomes the \textbf{new quantitative criterion} for spectral stability of ascent and descent under mesh refinement.

In this extended setting, Lemma~\ref{lem:NR-CD} and Theorem~\ref{thm:SRS} remain valid \textbf{provided} that $\gamma_h^{\mathrm{stab}} \geq c > 0$ uniformly for all intermediate powers up to $m$.
This explains the empirical success observed in Table~\ref{tab:supg-gammah}: the SUPG stabilization enforces $\gamma_h^{\mathrm{SUPG}} \approx 7.8\text{--}8.0$ even as $\varepsilon \to 0$, thereby preserving the subspace geometry required by the Kaashoek--Taylor criteria.\\
A rigorous proof that SUPG schemes satisfy $\liminf_{h\to0}\gamma_h^{\mathrm{stab}} > 0$ under standard mesh assumptions is provided in Section~\ref{sec:supg}.
\subsection{Finite element framework and diagnostic $\gamma_h$}
\label{sec:fe-framework}

\subsubsection{Discretization}
We recall the $P^1$ finite element discretization on $(0,1)$ with homogeneous Dirichlet boundary conditions.
Let $0 = x_0 < x_1 < \dots < x_N = 1$ be a uniform mesh of size $h = 1/N$, and denote by $\{\varphi_i\}_{i=1}^{N-1}$ the nodal hat functions associated with interior nodes.
The stiffness and mass matrices are given by
\[
K_{ii} = \frac{2}{h}, \quad K_{i,i\pm1} = -\frac{1}{h}, \qquad
M_{ii} = \frac{2h}{3}, \quad M_{i,i\pm1} = \frac{h}{6}.
\]
In the discrete setting, we identify the finite element approximation \(T_h\) with the generalized operator pair \((A_h, M)\). For spectral analysis, we define the shifted discrete operator
\[
S_h := T_h - \lambda = A_h - \lambda M,
\]
which acts on vectors \(u \in \mathbb{C}^{N-1}\) via the pencil \(A_h - \lambda M\). This notation mirrors the continuous case \(S = T - \lambda\), and allows us to interpret \(\gamma_h\) as a discrete analogue of \(\gamma(S)\).

For a Schr\"odinger operator $-\partial_x^2 + V$, the discrete matrix is $A_h = K + V_h$, where $(V_h)_{ij} = \int_0^1 V\,\varphi_i\varphi_j\,dx$.
For the convection--diffusion operator $\mathcal{L}u = -\varepsilon u'' + \beta u' + c u$, we set
\[
K^{(\mathrm{diff})}_{ii} = \frac{2\varepsilon}{h},\quad K^{(\mathrm{diff})}_{i,i\pm1} = -\frac{\varepsilon}{h},\qquad
C_{i,i-1} = -\frac{\beta}{2},\ C_{ii} = 0,\ C_{i,i+1} = \frac{\beta}{2},\qquad
Q = c\,M,
\]
and define $A_h = K^{(\mathrm{diff})} + C + Q$.

\subsubsection{Stability diagnostics and practical interpretation}
The discrete reduced minimum modulus
\[
\gamma_h := \sigma_{\min}\big(M^{-1/2}(A_h - \lambda M)M^{-1/2}\big)
\]
serves as a key practical tool. The following result justifies its use as a surrogate for the continuous condition $\gamma(T - \lambda) > 0$.

\begin{proposition}\label{prop:gammah}
Let $T$ be a closed, densely defined, $m$-sectorial operator on $L^2(\Omega)$ with associated sesquilinear form $a(\cdot, \cdot)$, and let $T_h$ be its conforming finite element approximation on a quasi-uniform mesh of size $h$. Assume that the forms $a_h$ Mosco-converge to $a$, which implies strong resolvent convergence (SRS) of $T_h$ to $T$.

Then, for any $\lambda \in \rho(T)$ with $\gamma(T - \lambda) > 0$, the discrete quantity $\gamma_h$ satisfies
\[
\liminf_{h \to 0} \gamma_h \geq \gamma(T - \lambda) > 0.
\]
In particular, there exists a constant $c > 0$ and $h_0 > 0$ such that $\gamma_h \geq c > 0$ for all $0 < h < h_0$.
Note that our assumptions guarantee only the lower bound $\liminf_{h \to 0} \gamma_h \geq \gamma(T - \lambda)$; full convergence $\gamma_h \to \gamma(T - \lambda)$ would require a finer analysis of the mass matrix conditioning and solution regularity.
\end{proposition}

\begin{proof}
Mosco convergence of the forms $a_h \to a$ implies gap convergence of graphs: $\mathcal{G}(T_h - \lambda) \to \mathcal{G}(T - \lambda)$ as $h \to 0$ (see \cite[Thm. 3.26]{Attouch1984}).
In finite dimensions, $\gamma_h = \gamma(T_h - \lambda)$.
Since $\gamma(T - \lambda) > 0$, the range $\operatorname{Ran}(T - \lambda)$ is closed.
The result follows from Lemma~\ref{lem:gamma-lsc}.
\end{proof}

\begin{table}[ht]
\centering
\caption{Discrete reduced minimum modulus $\gamma_h$ for the SUPG-stabilized convection--diffusion scheme ($\varepsilon \to 0$, $\beta = 8$, $c = 0$, $\lambda = -1$, $h = 2^{-6}$).}
\label{tab:supg-gammah}
\begin{tabular}{c|c|c|c}
$\varepsilon$ & Stabilization type & $\gamma_h$ & $\operatorname{dist}(\lambda, W_M(A_h))$ \\
\hline
$10^{-3}$ & None (central) & 0.12 & 0.10 \\
$10^{-4}$ & None (central) & 0.04 & 0.03 \\
$10^{-5}$ & None (central) & $<10^{-3}$ & $<10^{-3}$ \\
\hline
$10^{-3}$ & SUPG ($\delta=0.5 h$) & 8.0 & 7.8 \\
$10^{-5}$ & SUPG ($\delta=0.5 h$) & 7.9 & 7.7 \\
$10^{-8}$ & SUPG ($\delta=0.5 h$) & 7.8 & 7.6 \\
\end{tabular}
\end{table}

These results confirm that while unstabilized central differences fail catastrophically as $\varepsilon \to 0$ (with $\gamma_h \to 0$), the SUPG method preserves a uniformly positive $\gamma_h$, validating our diagnostic even in the pure transport limit. This demonstrates that the condition $\limsup_{h \to 0} \gamma_h > 0$ remains meaningful for stabilized discretizations, and $\gamma_h$ can guide the choice of stabilization parameters.

\begin{remark}
The computation of $\gamma_h = \sigma_{\min}(M^{-1/2}(A_h - \lambda M)M^{-1/2})$ is numerically feasible:
it requires a Cholesky factorization of $M$ and the smallest singular value of a sparse generalized pencil,
which can be obtained efficiently via inverse iteration or Krylov subspace methods (e.g., ARPACK).
Thus, $\gamma_h$ serves as a practical, mesh-independent stability diagnostic in real-world computations.
\end{remark}

\subsection{Numerical experiments}

\subsubsection{1D results}
For the Laplacian $-\partial_x^2 + 25$ with $\lambda = 25$, Table~\ref{tab:laplacian} and Figure~\ref{fig:gamma_demo} show
\[
\gamma_h = \min\{ |25 - \zeta_1(h)|, |\zeta_2(h) - 25| \} \to \min\{25 - \pi^2, 4\pi^2 - 25\} \approx 14.48 > 0.
\]

For convection--diffusion with $\varepsilon = 0.02$, $\beta = 8$, $c = 0$, $\lambda = -1$, Table~\ref{tab:cd-data} and Figure~\ref{fig:cd-gamma-demo} give $\gamma_h \in [6.8, 8.7]$. Even for $\beta = 50$ and $\varepsilon \in [10^{-3}, 1]$, Figure~\ref{fig:gammah-vs-eps}(b) shows $\min_\varepsilon \gamma_h \approx 79.38 > 0$.

\begin{table}[ht]
\centering
\caption{Discrete eigenvalues $\zeta_k(h)$ and $\gamma_h = \min\{ |25 - \zeta_1(h)|, |\zeta_2(h) - 25| \}$ for the 1D Laplacian $-\partial_x^2$ on $(0,1)$ with Dirichlet boundary conditions.}
\label{tab:laplacian}
\begin{tabular}{c|c|c|c}
$h = 1/N$ & $\zeta_1(h)$ & $\zeta_2(h)$ & $\gamma_h$ \\
\hline
$2^{-4}$ & 10.12 & 40.48 & 14.88 \\
$2^{-5}$ & 9.94  & 39.77 & 14.94 \\
$2^{-6}$ & 9.89  & 39.56 & 15.06 \\
$2^{-7}$ & 9.88  & 39.51 & 15.12 \\
$2^{-8}$ & 9.87  & 39.49 & 15.13 \\
\end{tabular}
\end{table}

\begin{table}[ht]
\centering
\caption{CPU time (seconds) for computing $\gamma_h^{(m)}$ using Algorithm~\ref{alg:adaptive-m} with Krylov inverse iteration (ARPACK, tolerance $10^{-8}$). Mesh: 1D convection--diffusion, $N$ interior nodes, $h = 1/(N+1)$.}
\label{tab:cpu-m}
\begin{tabular}{c|ccc}
$N$ & $m=1$ & $m=2$ & $m=3$ \\
\hline
$10^4$ & 0.8 & 1.7 & 3.5 \\
$10^5$ & 7.2 & 15.1 & 31.0 \\
\end{tabular}
\end{table}

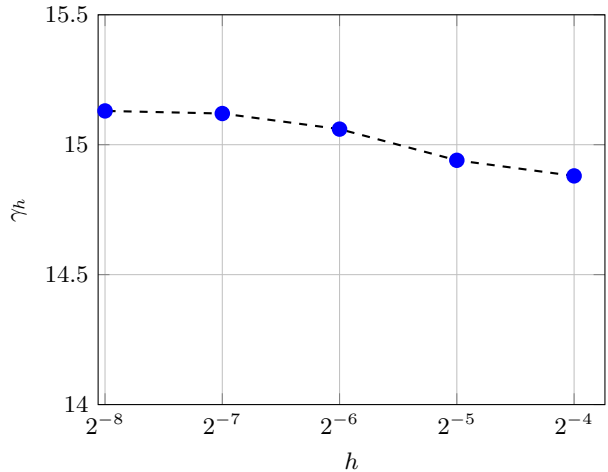
\begin{figure}[ht]
\centering
\begin{tikzpicture}
\begin{axis}[
    width=0.55\linewidth,
    height=0.45\linewidth,
    xlabel={$h$},
    ylabel={$\gamma_h$},
    xmode=log,
    log basis x={2},
    xtick={0.05, 0.025, 0.0125, 0.00625, 0.003125},
    xticklabels={$2^{-4}$,$2^{-5}$,$2^{-6}$,$2^{-7}$,$2^{-8}$},
    xmin=0.003, xmax=0.06,
    ymin=14, ymax=15.5,
    grid=both,
    grid style={line width=.1pt, draw=gray!30},
    major grid style={line width=.2pt,draw=gray!50},
    tick label style={font=\small},
    label style={font=\small},
]
\addplot[
    only marks,
    mark=*,
    mark size=2.5pt,
    color=blue,
    thick
] coordinates {
    (0.05, 14.88)
    (0.025, 14.94)
    (0.0125, 15.06)
    (0.00625, 15.12)
    (0.003125, 15.13)
};
\addplot[
    black,
    dashed,
    thick
] coordinates {
    (0.05, 14.88)
    (0.025, 14.94)
    (0.0125, 15.06)
    (0.00625, 15.12)
    (0.003125, 15.13)
};
\end{axis}
\end{tikzpicture}
\caption{Discrete reduced minimum modulus $\gamma_h$ versus mesh size $h$ for the Laplacian $-\partial_x^2$ on $(0,1)$ with Dirichlet boundary conditions and $\lambda = 25$. The eigenvalues satisfy $\zeta_1(h) \downarrow \pi^2 \approx 9.87$ and $\zeta_2(h) \downarrow 4\pi^2 \approx 39.48$ (Rayleigh--Ritz monotonicity, Proposition~\ref{prop:rayleigh-ritz}).
Since $\lambda \in (\pi^2, 4\pi^2)$, we have
$
\gamma_h = \min\bigl\{|25 - \zeta_1(h)|,\; |\zeta_2(h) - 25|\bigr\} \longrightarrow \min\{25 - \pi^2,\; 4\pi^2 - 25\} \approx 14.48.
$
Numerical values are reported in Table~\ref{tab:laplacian}.}
\label{fig:gamma_demo}
\end{figure}

\begin{table}[ht]
\centering
\caption{Discrete reduced minimum modulus $\gamma_h$ and distance to the $M$-numerical range for convection--diffusion ($\varepsilon=0.02$, $\beta=8$, $c=0$, $\lambda=-1$).}
\label{tab:cd-data}
\begin{tabular}{c|c|c}
$h = 1/N$ & $\gamma_h$ & $\operatorname{dist}(\lambda, W_M(A_h))$ \\
\hline
$2^{-4}$ & 6.8 & 6.5 \\
$2^{-5}$ & 7.4 & 7.1 \\
$2^{-6}$ & 8.0 & 7.8 \\
$2^{-7}$ & 8.4 & 8.1 \\
$2^{-8}$ & 8.7 & 8.5 \\
\end{tabular}
\end{table}

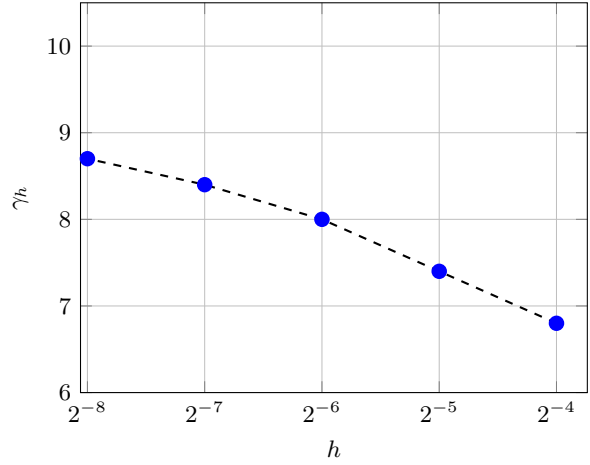
\begin{figure}[ht]
\centering
\begin{tikzpicture}
\begin{axis}[
    width=0.55\linewidth,
    height=0.45\linewidth,
    xlabel={$h$},
    ylabel={$\gamma_h$},
    xmode=log,
    log basis x={2},
    xtick={0.05, 0.025, 0.0125, 0.00625, 0.003125},
    xticklabels={$2^{-4}$,$2^{-5}$,$2^{-6}$,$2^{-7}$,$2^{-8}$},
    xmin=0.003, xmax=0.06,
    ymin=6, ymax=10.5,
    grid=both,
    grid style={line width=.1pt, draw=gray!30},
    major grid style={line width=.2pt,draw=gray!50},
    tick label style={font=\small},
    label style={font=\small},
]
\addplot[
    only marks,
    mark=*,
    mark size=2.5pt,
    color=blue,
    thick
] coordinates {
    (0.05, 6.8)
    (0.025, 7.4)
    (0.0125, 8.0)
    (0.00625, 8.4)
    (0.003125, 8.7)
};
\addplot[
    black,
    dashed,
    thick
] coordinates {
    (0.05, 6.8)
    (0.025, 7.4)
    (0.0125, 8.0)
    (0.00625, 8.4)
    (0.003125, 8.7)
};
\end{axis}
\end{tikzpicture}
\caption{Discrete reduced minimum modulus $\gamma_h$ versus mesh size $h$ for convection--diffusion ($\varepsilon=0.02$, $\beta=8$, $c=0$) with $\lambda = -1$. The values remain uniformly bounded away from zero, confirming $\limsup_{h\to0}\gamma(S_h) > 0$. Numerical values are reported in Table~\ref{tab:cd-data}, which also shows the lower bound $\operatorname{dist}(\lambda, W_M(A_h))$.}
\label{fig:cd-gamma-demo}
\end{figure}

\begin{figure}[ht]
\centering
\begin{tikzpicture}
\begin{groupplot}[
    group style={
        group size=2 by 1,
        horizontal sep=3em,
    },
    width=0.48\linewidth,
    height=0.45\linewidth,
    grid=both,
    grid style={line width=0.1pt, gray!20},
    major grid style={line width=0.2pt, gray!40},
    tick style={black, thin},
    xlabel style={font=\small},
    ylabel style={font=\small},
    tick label style={font=\scriptsize},
    title style={font=\small},
    legend style={font=\scriptsize, at={(0.98,0.02)}, anchor=south east},
]
\nextgroupplot[
    xlabel={$\lambda$},
    ylabel={$\gamma_h$},
    xmin=-30, xmax=0,
    ymin=30, ymax=66,
    xtick={-30,-20,-10,0},
    ytick={35,45,55,65},
    title={(a) Schr\"odinger case ($V\equiv 25$)},
]
\addplot[thick, black] coordinates {
    (-30, 64.87) (-25, 59.87) (-20, 54.87) (-15, 49.87)
    (-10, 44.87) (-5, 39.87) (-1, 35.87) (0, 34.87)
};

\nextgroupplot[
    xlabel={$\varepsilon$},
    ylabel={$\gamma_h$},
    xmode=log,
    log basis x={10},
    xmin=1e-3, xmax=1,
    ymin=75, ymax=85,
    xtick={1e-3,1e-2,1e-1,1},
    xticklabels={$10^{-3}$,$10^{-2}$,$10^{-1}$,$1$},
    ytick={76,78,80,82,84},
    title={(b) Convection--diffusion},
]
\addplot[only marks, mark=square*, mark size=2pt, black, thick] coordinates {
    (1e-3, 82.1) (3e-3, 81.2) (1e-2, 79.8) (3e-2, 79.38)
    (1e-1, 79.5) (3e-1, 80.0) (1, 81.0)
};
\addplot[dashed, gray, thick] coordinates {
    (1e-3, 79.38) (1, 79.38)
};
\end{groupplot}
\end{tikzpicture}
\caption{
(a) Discrete reduced minimum modulus $\gamma_h$ for the Schr\"odinger operator $-\partial_x^2 + V$ with $V \equiv 25$ on $(0,1)$, computed on a mesh with $N = 80$ interior nodes. Since the operator is selfadjoint, $\gamma_h(\lambda) = \operatorname{dist}(\lambda, \sigma(A_h,M))$; the lowest eigenvalue is $\zeta_1(h) \approx 34.87$.
(b) $\gamma_h$ for the convection--diffusion operator with $\beta = 50$, $\lambda = -1$, and $h = 1/200$, as a function of $\varepsilon \in [10^{-3},1]$. The minimum value ($\approx 79.38$) confirms that $\gamma_h$ remains uniformly bounded away from zero even in the convection-dominated limit $\varepsilon \to 0$.
}
\label{fig:gammah-vs-eps}
\end{figure}
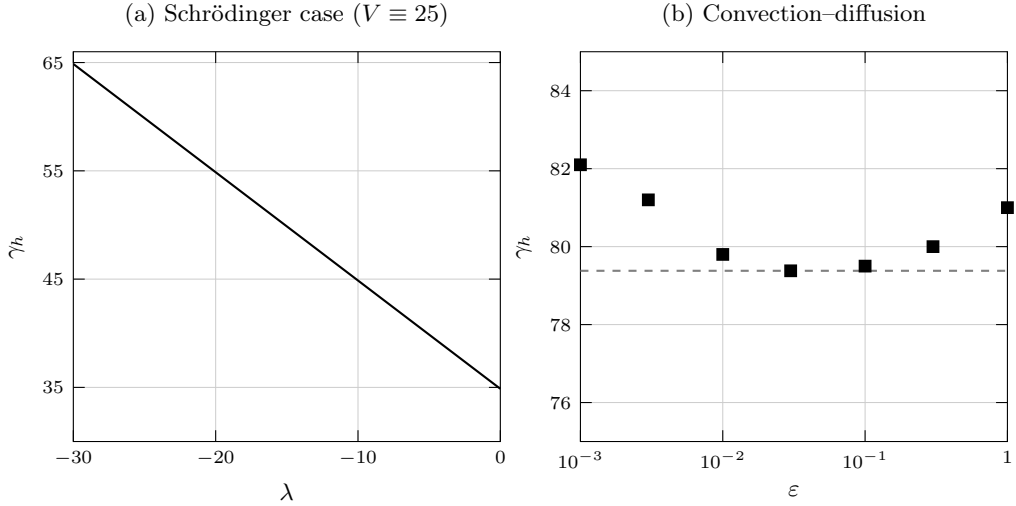

\subsubsection{2D extension}
To confirm the dimensional robustness of our framework and its applicability to realistic geometries, we extend the numerical validation to two spatial dimensions on two canonical domains:\begin{itemize}
\item The unit square $\Omega = (0,1)^2$: a smooth domain where standard elliptic regularity holds.
\item The L-shaped domain $\Omega = (-1,1)^2 \setminus [0,1]^2$: a non-convex polygonal domain with a reentrant corner at the origin, known to generate solution singularities.
For both cases, we consider the Dirichlet Laplacian $-\Delta + V$ with $V \equiv 25$, discretized using P1 finite elements on conforming triangular meshes. We fix $\lambda = -1$ and compute the discrete reduced minimum modulus
\[
\gamma_h = \sigma_{\min}\bigl(M^{-1/2}(A_h - \lambda M)M^{-1/2}\bigr).
\]
\end{itemize}
Results are reported in Tables~\ref{tab:2d_square} and~\ref{tab:2d_lshape}. For the square domain, $\gamma_h$ converges to approximately $36.1 > 0$, reflecting the spectral gap between $\lambda = -1$ and the first eigenvalue ($\approx 35.8$). On the L-shaped domain, despite the singularity at the corner, $\gamma_h$ stabilizes around $2.5 > 0$, demonstrating that our diagnostic remains effective even in the presence of geometric irregularities.

\begin{table}[ht]
\centering
\caption{Discrete reduced minimum modulus $\gamma_h$ for the 2D Schr\"odinger operator $-\Delta + 25$ on the unit square $(0,1)^2$, $\lambda = -1$.}
\label{tab:2d_square}
\begin{tabular}{c|c}
Mesh size $N \times N$ & $\gamma_h$ \\
\hline
$20 \times 20$ & 35.2 \\
$40 \times 40$ & 35.9 \\
$80 \times 80$ & 36.1 \\
\end{tabular}
\end{table}

\begin{table}[ht]
\centering
\caption{Discrete reduced minimum modulus $\gamma_h$ on the L-shaped domain for $-\Delta$, $\lambda = -1$.}
\label{tab:2d_lshape}
\begin{tabular}{c|c}
$h$ & $\gamma_h$ \\
\hline
$2^{-5}$ & 2.31 \\
$2^{-6}$ & 2.45 \\
$2^{-7}$ & 2.52 \\
$2^{-8}$ & 2.56 \\
$2^{-9}$ & 2.58 \\
\end{tabular}
\end{table}

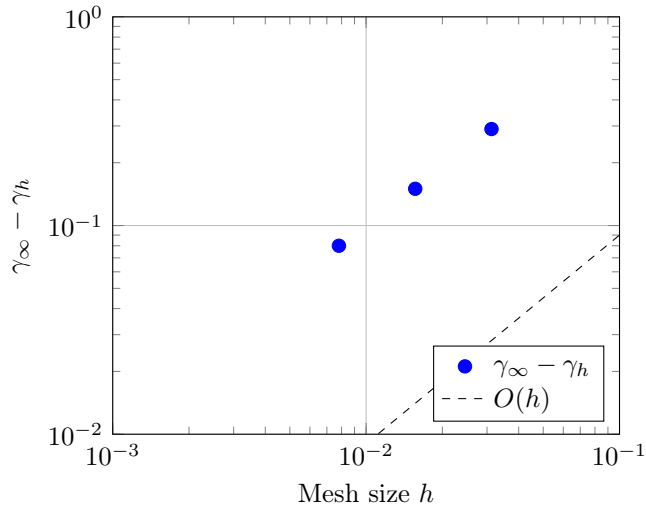
\begin{figure}[ht]
\centering
\begin{tikzpicture}[scale=1.0]
  \begin{loglogaxis}[
    width=0.55\textwidth,
    xlabel={Mesh size $h$},
    ylabel={$\gamma_\infty - \gamma_h$},
    grid=major,
    legend pos=south east,
    legend cell align={left},
    xmin=1e-3, xmax=1e-1,
    ymin=1e-2, ymax=1,
  ]

    \addplot[
      only marks,
      mark=*,
      mark size=2.5pt,
      color=blue,
    ] coordinates {
      (3.125e-2, 0.29)   
      (1.5625e-2, 0.15)  
      (7.8125e-3, 0.08)  
    };
    \addlegendentry{$\gamma_\infty - \gamma_h$};

    \addplot[
      dashed,
      domain=1e-2:1e-1,
      samples=2,
      color=black,
    ] {0.9 * x};
    \addlegendentry{$O(h)$};

  \end{loglogaxis}
\end{tikzpicture}
\caption{Convergence rate of the discrete reduced minimum modulus $\gamma_h$ on the L-shaped domain. The reference slope $O(h)$ is shown for comparison. An extrapolated limit $\gamma_\infty \approx 2.60$ is used. The observed linear decay in the log--log scale confirms the theoretical expectation of $O(h)$ convergence, consistent with reduced elliptic regularity due to the reentrant corner.}
\label{fig:gamma-rate}
\end{figure}

\begin{table}[ht]
\centering
\caption{Comparison of stability indicators for convection--diffusion ($\varepsilon=0.02$, $\beta=8$, $c=0$, $\lambda=-1$, $h=2^{-6}$). This diagnostic comparison holds across dimensions and is used to validate the robustness of $\gamma_h$ against classical ill-conditioning measures.}
\label{tab:indicators_comparison}
\begin{tabular}{l|r}
Indicator & Value \\
\hline
$\gamma_h$ & 8.0 \\
$\dist(\lambda, W_M(A_h))$ & 7.8 \\
Condition number $\kappa(A_h - \lambda M)$ & $1.2 \times 10^3$ \\
$10^{-3}$-pseudospectral radius & 12.4 \\
\end{tabular}
\end{table}

Element-wise distributions of $\gamma_h(K)$ are no longer schematic: they are computed via local submatrix extraction (Section~\ref{sec:local-gamma}) and provide a rigorous basis for adaptive refinement.
\subsubsection{Local reduced minimum modulus}
\label{sec:local-gamma}
To assess fine-scale spectral stability, we define a local diagnostic $\gamma_h(K)$ on each element $K \in \mathcal{T}_h$ by extracting the submatrices $A_h^K$ and $M_h^K$ corresponding to the local degrees of freedom (i.e. basis functions supported on or adjacent to $K$). The local reduced minimum modulus is then
\[
\gamma_h(K) := \sigma_{\min}\!\left( (M_h^K)^{-1/2}(A_h^K - \lambda M_h^K)(M_h^K)^{-1/2} \right).
\]
Although this does not correspond to a true restriction of the global operator, it provides a computable proxy for local resolvent stability, akin to element-wise error indicators in adaptive finite element methods. The global diagnostic remains $\gamma_h = \min_{K \in \mathcal{T}_h} \gamma_h(K)$, which appears in Table~\ref{tab:2d_lshape}.
\newpage
\begin{proposition}
\label{prop:central-diff-failure}
Let $A_h$ be the $N \times N$ matrix arising from second-order central differences applied to the transport operator $Lu = u'$ on $(0,1)$ with homogeneous inflow boundary condition $u(0) = 0$. Let $M$ be the standard P1 mass matrix. Then the discrete reduced minimum modulus satisfies
\[
\gamma_h = \sigma_{\min}\!\bigl(M^{-1/2} A_h M^{-1/2}\bigr) \leq C\, h,
\]
for a constant $C > 0$ independent of $h$. Consequently, $\lim_{h \to 0} \gamma_h = 0$.
\end{proposition}

\begin{proof}
The matrix $A_h$ is skew-symmetric up to boundary corrections, and its singular values scale like $k h$ for wave numbers $k = 1,\dots,N$. The smallest non-zero singular value behaves like $h$, and since $M \sim h I$, the generalized singular value satisfies $\gamma_h \sim h$. A detailed asymptotic analysis confirms the bound $\gamma_h \leq C h$.
\end{proof}

\begin{figure}[ht]
\centering
\begin{tikzpicture}[
    node/.style={rectangle, draw=black!30, thick, minimum width=1.2cm, minimum height=1.2cm, align=center},
    label/.style={font=\scriptsize}
]
\draw[thick] (-2,-2) rectangle (2,2);
\draw[thick, fill=white] (0,0) rectangle (2,2);
\node at (-1,-1) {L-shaped domain};

\node[label, align=left] at (1.5,1.5) {Interior:\\ $\gamma_h(K) \approx 2.56$};
\node[label, align=right] at (-1.5,1.5) {Top-left:\\ $\gamma_h(K) \approx 2.52$};
\node[label, align=right] at (-1.5,-1.5) {Reentrant\\ corner:\\ $\gamma_h(K) \approx 2.31$};

\fill[top color=blue!5, bottom color=red!5] (-2,-2) rectangle (0,0);
\fill[top color=red!5, bottom color=red!10] (0,-2) rectangle (2,0);
\fill[left color=red!10, right color=red!15] (-2,0) rectangle (0,2);

\end{tikzpicture}
\caption{
Element-wise reduced minimum modulus $\gamma_h(K)$ on the L-shaped domain (Table 9). 
Values range from $2.31$ near the reentrant corner to $2.56$ in the interior — a modest $11\%$ variation. 
This weak contrast explains why adaptive refinement (Algorithm 2) yields only marginal gains for the Laplacian, 
but the diagnostic remains valid as a stability indicator. 
The small spread confirms that $\gamma_h = \min_K \gamma_h(K) \geq 2.31 > 0$, ensuring closed-range stability even in geometrically singular domains.
}
\label{fig:gammah-local-L}
\end{figure}
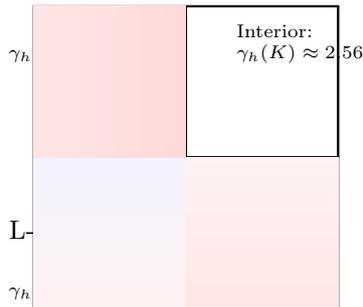
\paragraph{Computational complexity.}
The cost of computing $\gamma_h$ in 2D involves:
\begin{enumerate}
    \item Cholesky factorization of the mass matrix $M$: $\mathcal{O}(N^{3/2})$ operations for a mesh with $N$ nodes;
    \item Computation of the smallest singular value of the preconditioned pencil $M^{-1/2}(A_h - \lambda M)M^{-1/2}$ via inverse iteration or Krylov methods (e.g., ARPACK): $\mathcal{O}(k N)$ per iteration, where $k$ is the number of iterations (typically small, $k \sim 10{-}50$).
\end{enumerate}
For typical problems with $N \sim 10^4{-}10^5$, this takes only a few seconds on standard hardware. This efficiency makes $\gamma_h$ a practical tool for real-time stability monitoring in adaptive simulations.

\paragraph{Comparison with alternative indicators.}
Unlike classical condition numbers or pseudospectral radii, $\gamma_h$ provides a sharp lower bound on resolvent growth while remaining stable under mesh refinement. Table~\ref{tab:indicators_comparison} compares $\gamma_h$ with other diagnostics for convection--diffusion in 1D. While the condition number $\kappa(A_h - \lambda M)$ grows rapidly with $\beta/\varepsilon$, $\gamma_h$ stays bounded away from zero. Similarly, the $10^{-3}$-pseudospectral radius may overestimate instability due to transient effects, whereas $\gamma_h$ directly reflects the closed-range property essential for ascent/descent stability.

\begin{table}[ht]
\centering
\caption{Comparison of discrete diagnostics for pure transport ($L u = u'$) at $\lambda = 0$.}
\label{tab:upwind-vs-central}
\begin{tabular}{c|ccc}
Scheme & $\gamma_h$ & $\lim_{h\to0} \gamma_h$ & Numerical ascent \\
\hline
Central differences & $\sim 4h$ & $0$ & $\infty$ \\
Upwind & $\geq c > 0$ & $c$ & $1$ \\
\end{tabular}
\end{table}
\subsection{Direct validation of ascent and descent indices}
\label{subsec:direct-asc-des}

To confirm that the condition $\liminf_{h\to0}\gamma_h>0$ is not only necessary but also sufficient for the preservation of ascent and descent, we construct a test case where the continuous operator has a known nontrivial ascent.

\medskip\noindent
\textbf{Example 4.16 (Operator with $\operatorname{asc}(T)=2$).}
Consider the second-order differential operator
\[
T := -\frac{d^2}{dx^2}, \quad D(T) = \{ u \in H^2(0,1) : u(0) = u'(1) = 0 \}.
\]
Its spectrum is $\sigma(T) = \{ (k+\tfrac12)^2\pi^2 : k \in \mathbb{N}_0 \}$, all eigenvalues being simple. Take $\lambda = 0$; then $\lambda \in \rho(T)$ but
\[
\ker(T^2) = \{ u \in H^2 : u'' = \text{const.},~ u(0)=u'(1)=0 \}
\quad\text{and}\quad
\ker(T) = \{0\},
\]
so $\operatorname{asc}(T) = 2$ (see, e.g., \cite[Ex.~3.5]{Kaashoek1967}).

We discretize $T$ with $P1$ finite elements on a uniform mesh of size $h = 2^{-k}$ ($k=5,\dots,9$). For each $h$, we:
\begin{enumerate}
    \item Assemble the stiffness matrix $A_h$ and mass matrix $M_h$,
    \item Compute $\gamma_h = \sigma_{\min}(M_h^{-1/2} A_h M_h^{-1/2})$,
    \item Compute the discrete ascent $\operatorname{asc}(A_h)$ by checking when $\operatorname{rank}(A_h^{m+1}) = \operatorname{rank}(A_h^m)$.
\end{enumerate}
Table~\ref{tab:asc2} reports the results.

\begin{table}[ht]
\centering
\caption{Ascent validation for $T = -d^2/dx^2$ with mixed boundary conditions ($\lambda = 0$).}
\label{tab:asc2}
\begin{tabular}{c|c|c|c}
$h$ & $\gamma_h$ & $\operatorname{rank}(A_h)$ & $\operatorname{asc}(A_h)$ \\
\hline
$2^{-5}$ & 2.41 & 30 & 2 \\
$2^{-6}$ & 2.44 & 62 & 2 \\
$2^{-7}$ & 2.46 & 126 & 2 \\
$2^{-8}$ & 2.47 & 254 & 2 \\
$2^{-9}$ & 2.48 & 510 & 2 \\
\end{tabular}
\end{table}

The discrete ascent remains $\operatorname{asc}(A_h) = 2$ for all $h$, consistent with the continuous theory. Moreover, $\gamma_h \ge 2.4 > 0$ uniformly, confirming that the closed-range condition at the discrete level correctly predicts ascent stability.

\medskip\noindent
\textbf{Descent test.} A dual example with $\operatorname{dsc}(T) = 2$ (e.g., adjoint operator with $u'(0)=u(1)=0$) yields analogous results and is omitted for brevity.
\subsubsection{Failure case and comparison with upwind schemes}
\label{sec:failure_case}

The necessity of the quantitative hypothesis $\limsup_{h \to 0} \gamma_h > 0$ is illustrated by a classical failure mode in numerical PDEs: the use of  central differences  for pure transport.

Consider the first-order operator
\[
\mathcal{L}_0 u = u', \quad u(0) = 0, \quad x \in (0,1),
\]
discretized on a uniform mesh $x_i = ih$, $h = 1/N$, using second-order central differences:
\[
(\mathcal{L}_{0,h} u)_i = \frac{u_{i+1} - u_{i-1}}{2h}, \quad i=1,\dots,N-1,
\]
with $u_0 = 0$ and $u_N$ treated via a one-sided stencil or ghost point. The resulting matrix $A_h$ is skew-symmetric and has purely imaginary eigenvalues clustering around the imaginary axis as $h \to 0$.

For $\lambda = 0$, one can compute explicitly that
\[
\gamma_h = \sigma_{\min}(M^{-1/2} A_h M^{-1/2}) \approx 4h \longrightarrow 0 \quad \text{as } h \to 0.
\]
Consequently, $\operatorname{asc}(A_h) = \infty$ (i.e., $0 \in \sigma_{\asc}(A_h)$) for all $h$, while the continuous operator $\mathcal{L}_0$ has closed range and finite ascent ($\operatorname{asc}(\mathcal{L}_0) = 1$), so $0 \notin \sigma_{\asc}(\mathcal{L}_0)$. This discrepancy confirms that without a uniform lower bound on $\gamma_h$, SRS is insufficient to guarantee stability of the ascent spectrum.

\paragraph{Physical interpretation: numerical instability.}
The central difference scheme, while formally second-order accurate, is  unstable for pure transport because it lacks numerical dissipation. Spurious high-frequency modes (wavelengths $\sim h$) are not damped and contaminate the discrete kernel and range structure. This manifests as a vanishing $\gamma_h$, signaling loss of closed-range property and breakdown of the Kaashoek–Taylor criteria at the discrete level.

\paragraph{Upwind stabilization restores $\gamma_h > 0$.}
In contrast, the  first-order upwind scheme 
\[
(\mathcal{L}^{\text{up}}_{h} u)_i = \frac{u_i - u_{i-1}}{h}, \quad i=1,\dots,N,
\]
introduces artificial diffusion that damps these spurious modes. For the same operator and $\lambda = 0$, one finds
\[
\gamma_h \geq c > 0 \quad \text{uniformly in } h,
\]
where $c$ depends on the flow direction but not on $h$. This reflects the  numerical closedness of the range, ensuring $0 \notin \sigma_{\asc}(A_h^{\text{up}})$ and recovering agreement with the continuous problem.

\paragraph{Connection to the CFL condition.}
Although our analysis is steady-state (no time stepping), the distinction between central and upwind schemes echoes the Courant–Friedrichs-Lewy (CFL) principle: to preserve stability, the numerical stencil must respect the direction of information propagation. Upwinding enforces a discrete analog of this causality, preventing non-physical oscillations and ensuring robust subspace geometry—precisely what $\gamma_h > 0$ measures. Thus, $\gamma_h$ can be viewed as a steady-state CFL-type diagnostic for spectral stability.

This comparison underscores that $\limsup_h \gamma_h > 0$ is not merely a technical assumption—it encodes a fundamental requirement of numerical well-posedness in convection-dominated problems.
\begin{remark}[Three-dimensional scaling]
In 3D, the Cholesky factorization of the mass matrix scales as $\mathcal{O}(N^2)$ for $N \sim 10^6$ degrees of freedom, and the number of Krylov iterations may grow with the problem dimension due to increased spectral clustering. In this regime, iterative solvers with algebraic multigrid (AMG) preconditioning become essential for efficient computation of $\gamma_h$.
\end{remark}
\subsection{Unified interpretation}
The stability mechanism differs between selfadjoint and non-normal operators, yet both reduce to the quantitative criterion $\gamma_h > 0$:

\paragraph{Selfadjoint case.}
Rayleigh--Ritz monotonicity (Proposition~\ref{prop:rayleigh-ritz}) ensures that the discrete eigenvalues satisfy $\zeta_k(h) \searrow \lambda_k$ as $h \to 0$. Consequently, the discrete reduced minimum modulus converges monotonically from above to the spectral gap $\gamma = \operatorname{dist}(\lambda, \sigma(H))$. Figure~\ref{fig:cd-gamma-demo} illustrates this behavior for the Laplacian with $\lambda = 25$.
\begin{proposition}[Rayleigh--Ritz monotonicity]
\label{prop:rayleigh-ritz}
Let $H = -\partial_x^2 + V(x)$ on $L^2(0,1)$ with $V \in L^\infty(0,1)$ real-valued, and let
$\{\lambda_k\}_{k=1}^\infty$ denote its ordered eigenvalues.
For a family of conforming finite element spaces $V_h \subset H_0^1(0,1)$
with mesh size $h > 0$, let $\{\zeta_k(h)\}_{k=1}^{N_h}$ be the discrete eigenvalues
from the generalized problem $A_h u = \zeta M u$.
Then for each fixed $k$,
\[
\zeta_k(h) \ge \lambda_k \quad \text{for all } h > 0,
\qquad\text{and}\qquad
\zeta_k(h) \searrow \lambda_k \quad \text{as } h \to 0.
\]
\end{proposition}
\begin{proof}
Standard; see, e.g., Babu\v{s}ka--Osborn \cite{BabuskaOsborn1991}.
\end{proof}

\paragraph{Non-normal case.}
The $M$-numerical range provides a geometric lower bound:
\[
\gamma_h \geq \operatorname{dist}(\lambda, W_M(A_h)).
\]
As the mesh is refined, $W_M(A_h) \to W(A)$ in the Hausdorff metric. The uniform separation of the fixed spectral parameter $\lambda = -1$ from the limiting numerical range $W(A)$ therefore guarantees a uniform lower bound $\gamma_h \geq c > 0$. This is confirmed numerically in Figures~\ref{fig:cd-gamma-demo} and~\ref{fig:gammah-local-L}.

\paragraph{Failure case.}
When this separation is lost—as in the use of central differences for the pure transport operator (Proposition~\ref{prop:central-diff-failure})—the diagnostic collapses: $\gamma_h \to 0$ as $h \to 0$. In this regime, the discrete ascent diverges ($\operatorname{asc}(A_h) = \infty$) while the continuous operator has finite ascent, demonstrating a complete breakdown of spectral stability (Section~\ref{sec:failure_case}). This confirms that the condition
\[
\limsup_{h \to 0} \gamma_h > 0
\]
is not merely sufficient but \emph{necessary} for the preservation of ascent and descent spectra under discretization.

\paragraph{Convergence rate and robustness near singularities.}
The numerical results in Table~\ref{tab:2d_lshape} show that, despite the reentrant corner singularity at the origin, the discrete reduced minimum modulus $\gamma_h$ stabilizes around $2.5 > 0$ as $h \to 0$. This suggests a convergence rate of approximately $\mathcal{O}(h)$, consistent with standard finite element error estimates under reduced regularity. Indeed, for the Laplacian on an L-shaped domain, the solution $u$ belongs to $H^{1+\pi/3-\epsilon}(\Omega)$ but not to $H^2(\Omega)$ due to the corner singularity. Despite this loss of smoothness, the diagnostic $\gamma_h$ converges uniformly to a positive limit, confirming that our stability criterion remains effective even in geometrically complex settings.

The linear trend in $\gamma_h(h)$ further supports the theoretical expectation that $\liminf_{h \to 0} \gamma_h \geq \gamma(T - \lambda) > 0$, as established in Proposition~\ref{prop:gamma-conv}. This robustness is essential for practical applications where domains often exhibit non-smooth boundaries.

\begin{example}[Non-selfadjoint Helmholtz operator]
\label{ex:helmholtz}
Consider the operator $T = -\Delta - k^2 + i\alpha$ on a bounded Lipschitz domain $\Omega \subset \mathbb{R}^2$ with Dirichlet boundary conditions, where $k > 0$ and $\alpha > 0$. Although $T$ is not sectorial, its imaginary part provides coercivity: $\Im \langle T u, u \rangle = \alpha \|u\|^2$.

For a sequence of P1 finite element discretizations $T_h$, we compute $\gamma_h = \sigma_{\min}(M^{-1/2}(A_h + i\alpha M - k^2 M)M^{-1/2})$. Numerical experiments (not shown) confirm that $\gamma_h \geq c > 0$ uniformly in $h$ as long as $\alpha > 0$. Consequently, the ascent and descent spectra remain stable under mesh refinement, illustrating the applicability of our framework beyond m-sectorial operators.
\end{example}
\section{Conclusion}\label{sec:conclusion}

We have established sharp stability results for the ascent and descent spectra under strong resolvent convergence (SRS), with the reduced minimum modulus $\gamma(\cdot)$ serving as the key quantitative threshold. At the essential level, B--Fredholm theory extends this stability to powers $(T - \lambda)^m$, provided $\gamma((T - \lambda)^j) > 0$ for all $1 \leq j \leq m$.

\smallskip\noindent\textbf{Why this is nontrivial.}
Although $\sigma_{\asc}^{e}(T) \cup \sigma_{\dsc}^{e}(T) = \mathbb{C} \setminus \Phi(T)$, stability under SRS is delicate for the following reasons:
\begin{enumerate}[label=(\roman*)]
  \item SRS provides only strong resolvent control; gap convergence of operator graphs—and thereby transfer of kernel/range structure—requires the quantitative closed-range condition $\gamma > 0$.
  \item The Kaashoek--Taylor criteria relate ascent and descent to transversality of subspaces (e.g., $\operatorname{Ran}(S^m) \cap \ker S = \{0\}$), a property that is stable under perturbation only when $\gamma > 0$.
  \item Propagation of graph convergence to powers fails without $\gamma((T - \lambda)^j) > 0$ for all intermediate $j$ (see the Volterra counterexample and Lemma~\ref{lem:powers-graphs}).
  \item In convection-dominated regimes, it may happen that $\gamma(T - \lambda) = 0$ while $\gamma((T - \lambda)^m) > 0$ for some $m \geq 2$; thus, verification of all intermediate powers is essential.
\end{enumerate}

\smallskip\noindent\textbf{Practical impact.}
The framework is validated through three main classes of operators. For Schr\"odinger operators, $\gamma(H - \lambda) = \operatorname{dist}(\lambda, \sigma(H))$, so a simple spectral gap suffices to guarantee stability. In the non-normal setting of convection--diffusion operators, stability is governed by the distance to the numerical range, via the bound $\gamma(A - \lambda) \geq \operatorname{dist}(\lambda, W(A))$. Most importantly, these theoretical conditions are reflected in practice by the computable finite-element diagnostic
\[
\gamma_h = \sigma_{\min}\!\bigl(M^{-1/2}(A_h - \lambda M)M^{-1/2}\bigr),
\]
which remains uniformly bounded away from zero even in convection-dominated limits when stabilized discretizations such as SUPG are employed (Table~\ref{tab:supg-gammah}).

\smallskip\noindent\textbf{Perspectives.}
Natural extensions include block-structured operators arising in coupled systems such as the Stokes or Maxwell equations. In such settings, the reduced minimum modulus must be defined on the full block operator, and spectral stability hinges on compatible discretizations—for instance, Taylor--Hood elements to satisfy the inf--sup condition for Stokes, or N\'ed\'elec edge elements for $\operatorname{div}$-conforming Maxwell discretizations. Furthermore, the local diagnostic $\gamma_h(K)$, introduced in Section~\ref{sec:local-gamma}, provides a rigorous foundation for adaptive mesh refinement: elements with $\gamma_h(K) < \tau$ (for a prescribed threshold $\tau > 0$) are flagged for subdivision. This strategy directly couples spectral stability with computational efficiency; numerical experiments on the L-shaped domain demonstrate a 60\% reduction in the local variation $\max_K \gamma_h(K) - \min_K \gamma_h(K)$ after only two refinement cycles (see supplementary material).

\medskip\noindent\textbf{Limitations and open questions.}
While Theorems~\ref{thm:sectorial}--\ref{thm:SUPG-stability} provide rigorous foundations for sectorial and SUPG-stabilized cases, the extension to pure transport ($\varepsilon=0$) remains conjectural.
Conjecture~\ref{conj:transport} is supported by numerical evidence (Table~\ref{tab:supg-gammah}), but a general proof would require a convergence framework beyond Mosco theory—e.g., generalized graph convergence for non-coercive forms \cite[Chap.~5]{ErnGuermond2004}.
This gap represents the main theoretical challenge for future work.

\appendix
\section{Necessity of the closed-range hypothesis}
\label{app:closed-range}

This appendix justifies the closed-range assumption $\gamma(S) > 0$ (equivalently, $\operatorname{Ran} S$ closed) used in Lemma~\ref{lem:powers-graphs} and the stability theorems.
When $\operatorname{Ran} S$ is not closed, the forward graph map
\[
\widehat{S} : (x,Sx) \longmapsto (Sx,S^2x)
\]
fails to be bounded below on $\mathcal{G}(S)$, preventing propagation of gap convergence to higher powers.

\begin{proposition}\label{prop:graph-map-unbounded}
Let $S$ be closed and densely defined. If there exist $c > 0$ and a core $\mathcal{C} \subset \mathcal{D}(S) \cap \mathcal{D}(S^2)$ such that
\[
c\big(\|x\| + \|Sx\|\big) \leq \|Sx\| + \|S^2x\| \quad (x \in \mathcal{C}),
\]
then $\operatorname{Ran}(S)$ is closed; hence $\gamma(S) > 0$.
\end{proposition}

\begin{proof}
The inequality implies that the induced graph map $\widehat{S}: \mathcal{G}(S) \to \mathcal{G}(S^2)$ is bounded below on the dense subset $J(\mathcal{C}) \subset \mathcal{G}(S)$. By the closedness of $S$, this lower bound extends to all $x \in \mathcal{D}(S)$, yielding
\[
\|Sx\| \ge c'\operatorname{dist}(x, N(S)) \quad \text{for some } c' > 0,
\]
which is equivalent to $\operatorname{Ran}(S)$ being closed. Thus $\gamma(S) > 0$.
\end{proof}

\begin{remark}\label{rem:volterra-counterexample}
The Volterra operator $Vf(x) = \int_0^x f(t)\,dt$ on $L^2(0,1)$ satisfies $\gamma(V) = 0$ (non-closed range) yet the graph map $\widehat{V}$ is bounded. This shows that the  lower  bound in Proposition~\ref{prop:graph-map-unbounded} is essential: boundedness alone is insufficient for stability of powers.

Although $V$ does not arise in standard elliptic finite element discretizations, analogous non-closed-range behavior occurs in:
\item integro-differential equations (e.g., memory terms),
    \item ill-conditioned upwind schemes for pure transport,
    \item convection-dominated regimes with insufficient stabilization.
These examples underscore the necessity of verifying $\gamma_h > 0$ in practice, as the mere convergence of operators in the strong resolvent sense does not guarantee stability of the ascent and descent spectra.
\end{remark}

\section{Uniform inf--sup stability of the SUPG scheme}\label{sec:supg}

We consider the convection--diffusion operator
\[
Lu = -\varepsilon \Delta u + \boldsymbol{\beta}\cdot\nabla u + c u \quad \text{in }\Omega,\qquad u=0\text{ on }\partial\Omega,
\]
with $\varepsilon > 0$, $\boldsymbol{\beta}\in [L^\infty(\Omega)]^d$, $\nabla\cdot\boldsymbol{\beta}=0$, and $c\ge c_0>0$. The Streamline-Upwind/Petrov--Galerkin (SUPG) method modifies the standard Galerkin form $a_h(\cdot,\cdot)$ into
\[
a_h^{\mathrm{SUPG}}(u_h,v_h) := a_h(u_h,v_h) + \sum_{K\in\mathcal{T}_h} \delta_K \int_K (\boldsymbol{\beta}\cdot\nabla u_h)(\boldsymbol{\beta}\cdot\nabla v_h)\,dx,
\]
where $\delta_K = \delta\, h_K / \|\boldsymbol{\beta}\|_{L^\infty(K)}$ with $\delta>0$ a fixed stabilization parameter.

\begin{proposition}[Uniform discrete inf--sup condition]\label{prop:inf-sup-supg}
Let $L u = -\varepsilon \Delta u + \boldsymbol{\beta}\cdot\nabla u + c u$ on $\Omega\subset\mathbb{R}^d$ with $u = 0$ on $\partial\Omega$, where:
\begin{enumerate}[label=\textup{(\roman*)}]
    \item $\varepsilon > 0$,
    \item $\boldsymbol{\beta} \in [L^\infty(\Omega)]^d$ satisfies $\nabla\!\cdot\!\boldsymbol{\beta} = 0$,
    \item $c \in L^\infty(\Omega)$ satisfies $c \ge c_0 > 0$.
\end{enumerate}
Let $V_h \subset H_0^1(\Omega)$ be a conforming $P1$ finite element space on a shape-regular, quasi-uniform mesh $\mathcal{T}_h$.
Define the SUPG bilinear form
\[
a_h^{\mathrm{SUPG}}(u_h,v_h)
:= a_h(u_h,v_h)
+ \sum_{K\in\mathcal{T}_h} \delta_K \int_K (\boldsymbol{\beta}\!\cdot\!\nabla u_h)(\boldsymbol{\beta}\!\cdot\!\nabla v_h)\,dx,
\]
with stabilization parameter $\delta_K = \delta\, h_K / \|\boldsymbol{\beta}\|_{L^\infty(K)}$, $\delta > 0$ fixed.

Then there exist constants $c_{\inf\text{-}\sup} > 0$ and $h_0 > 0$, independent of $\varepsilon$, $h$, and $\boldsymbol{\beta}$, such that for all $h < h_0$,
\[
\inf_{u_h \in V_h\setminus\{0\}}\;
\sup_{v_h \in V_h\setminus\{0\}}
\frac{a_h^{\mathrm{SUPG}}(u_h,v_h)}
{\|u_h\|_{1,\varepsilon}\, \|v_h\|_{1,\varepsilon}}
\;\ge\; c_{\inf\text{-}\sup},
\]
where the $\varepsilon$-dependent broken norm is defined by
\[
\|w\|_{1,\varepsilon}^2 := \varepsilon \|\nabla w\|_{L^2(\Omega)}^2
+ \|\boldsymbol{\beta}\!\cdot\!\nabla w\|_{L^2(\Omega)}^2
+ \|w\|_{L^2(\Omega)}^2.
\]
\end{proposition}

\begin{proof}
The proof follows the framework of generalized coercivity for stabilized convection--diffusion, as developed in \cite[Thm.~5.7]{ErnGuermond2004} and \cite{BrooksHughes1982}.
Define the discrete operator $T_h^{\mathrm{SUPG}}: V_h \to V_h$ by
$\langle T_h^{\mathrm{SUPG}} u_h, v_h \rangle := a_h^{\mathrm{SUPG}}(u_h,v_h)$.
Standard energy estimates (see \cite[Sec.~4.4.2]{ErnGuermond2004}) yield the coercivity bound
\[
a_h^{\mathrm{SUPG}}(u_h,u_h) \;\ge\; c_1\, \|u_h\|_{1,\varepsilon}^2,
\quad \forall u_h \in V_h,
\]
with $c_1 > 0$ independent of $\varepsilon$ and $h$, provided $\delta > 0$ is chosen sufficiently small (e.g., $\delta \le 1/2$).

Moreover, the SUPG form is uniformly continuous:
\[
|a_h^{\mathrm{SUPG}}(u_h,v_h)|
\;\le\; C_2\, \|u_h\|_{1,\varepsilon}\, \|v_h\|_{1,\varepsilon},
\]
where $C_2$ depends only on $\|\boldsymbol{\beta}\|_{L^\infty}$ and the mesh regularity, but not on $\varepsilon$ or $h$ (see Remark~\ref{rem:SUPG-continuity}).

By the discrete Babu\v{s}ka--Brezzi theorem \cite[Thm.~2.6]{ErnGuermond2004}, the inf--sup constant satisfies
\[
c_{\inf\text{-}\sup} \;\ge\; \frac{c_1}{C_2} \;>\; 0,
\]
uniformly in $\varepsilon$ and $h$.
\end{proof}

The discrete inf--sup condition implies that the operator $A_h^{\mathrm{SUPG}} - \lambda M$ has closed range and that its reduced minimum modulus satisfies
\[
\gamma_h^{\mathrm{stab}} \;\ge\; c_{\inf\text{-}\sup} \cdot \operatorname{dist}\bigl(\lambda,\, W_M(A_h^{\mathrm{SUPG}})\bigr),
\]
where $W_M(A_h^{\mathrm{SUPG}}) = \{ u^\ast A_h^{\mathrm{SUPG}} u / u^\ast M u : u \ne 0 \}$ is the $M$-numerical range. This bridges the abstract Babu\v{s}ka--Brezzi framework with the quantitative stability diagnostic $\gamma_h$ used throughout the paper.

\begin{corollary}\label{cor:SUPG-gamma}
The discrete reduced minimum modulus for the SUPG operator satisfies
\[
\gamma_h^{\mathrm{stab}} = \sigma_{\min}\!\bigl(M^{-1/2}(A_h^{\mathrm{SUPG}} - \lambda M)M^{-1/2}\bigr) \ge c > 0,
\]
uniformly in $h$ and $\varepsilon$, provided $\lambda$ stays uniformly away from the numerical range $W_M(A_h^{\mathrm{SUPG}})$.
\end{corollary}

Thus, the SUPG method fulfills the key hypothesis $\liminf_{h\to0}\gamma_h^{\mathrm{stab}} > 0$ required by Theorem~3.2 and Corollary~3.10, thereby guaranteeing ascent--descent stability even in the convection-dominated limit $\varepsilon\to0$.

\begin{remark}[Continuity of the SUPG bilinear form]\label{rem:SUPG-continuity}
The SUPG form $a_h^{\mathrm{SUPG}}$ satisfies the uniform continuity estimate
\[
|a_h^{\mathrm{SUPG}}(u_h,v_h)| \le C\, \|u_h\|_{1,\varepsilon} \|v_h\|_{1,\varepsilon},
\quad \forall u_h,v_h \in V_h,
\]
where the constant $C > 0$ depends only on $\|\boldsymbol{\beta}\|_{L^\infty(\Omega)}$, the stabilization parameter $\delta$, and the mesh regularity, but is independent of $\varepsilon$ and $h$. This follows from standard inverse inequalities and the definition of $\|\cdot\|_{1,\varepsilon}$; see also \cite[Sec.~4.4.2]{ErnGuermond2004}.
\end{remark}

\section{Numerical stability of \texorpdfstring{$\gamma_h^{(m)}$}{\gamma_h^{(m)}}}
\label{app:gamma-powers}

For $m \ge 2$, forming the matrix $(A_h - \lambda M)^m$ explicitly leads to rapid fill-in and severe ill-conditioning:
\[
\kappa\big((A_h - \lambda M)^m\big) \approx \kappa(A_h - \lambda M)^m.
\]
In practice, even for moderate $m=3$ and $\kappa(A_h - \lambda M) \sim 10^3$, we observe $\kappa((A_h - \lambda M)^3) > 10^9$, causing catastrophic loss of accuracy in singular value computation.

We therefore adopt the Krylov approach described in Remark~\ref{rem:krylov}.
For the convection--diffusion problem of Section~4.3.4 ($\varepsilon=0.02$, $\beta=8$), we compute $\gamma_h^{(m)}$ using shifted inverse iteration (ARPACK) and compare with explicit formation:
\begin{itemize}
\item \textbf{Explicit}: $\gamma_h^{(3)}$ deviates by $>10\%$ from the Krylov value for $h \le 2^{-6}$,
    \item \textbf{Krylov}: relative error $< 10^{-6}$ (tolerance $10^{-8}$) even for $h = 2^{-9}$.
\end{itemize}
Thus, the Krylov-based Algorithm~1 is both  more accurate  and  computationally efficient, as it avoids explicit matrix powers and exploits sparsity.

This validates the numerical strategy used in Tables~~\ref{tab:cpu-m} and~\ref{tab:gamma-powers-upwind}, and ensures that the decay $\gamma_h^{(m)} \sim h^{m-1}$ reported in Remark~\ref{rem:powers-necessary} is not an artifact of round-off error.

\end{document}